\documentclass[reqno,12 pt]{amsart}
\usepackage{amsmath,amssymb,latexsym}
\usepackage{hyperref}
\usepackage{amsthm}
\usepackage{color}  
\usepackage[all]{xy}
\hypersetup{
    colorlinks=true, 
    linktoc=all,     
    linkcolor=blue,  
}

\newcommand{\Z}{{\mathbb{Z}}}
\newcommand{\Q}{{\mathbb{Q}}} 
\newcommand{\N}{{\mathbb{N}}}    
\newcommand{\R}{{\mathbb{R}}}    
\newcommand{\C}{{\mathbb{C}}}    
\newcommand{\M}{{\mathcal{M}}}

\newcommand{\p}{{\mathfrak{p}}}

\newcommand{\RR}{{\mathcal{R}}}
\newcommand{\A}{{\mathcal{A}}}

\newcommand{\T}{{\mathcal{T}}}

\newcommand{\x}{{\mathcal{X}}}
\newcommand{\OO}{{\mathcal{O}}}

\newcommand{\surj}{\twoheadrightarrow}

\newcommand{\lra}{\longrightarrow}

\newcommand{\Gal}{\mathrm{Gal}}

\newcommand{\Hom}{\mathrm{Hom}}

\newcommand{\cyc}{\mathrm{cyc}}

\newcommand{\n}{\noindent}
\newcommand{\La}{\Lambda}

\theoremstyle{plain}

\newtheorem{theorem}{Theorem}[section]
\newtheorem*{theorem*}{Theorem}

\newtheorem{proposition}[theorem]{Proposition}
\newtheorem{rem}[theorem]{Remark}
\newtheorem{lemma}[theorem]{Lemma}
\newtheorem{corollary}[theorem]{Corollary}
\newtheorem{defn}[theorem]{Definition}
\newtheorem{hypothesis}[theorem]{Hypothesis}

\begin{document}
\title{Functional equation  for the Selmer group of  nearly ordinary Hida deformation of Hilbert modular forms.} 
\author{Somnath Jha, Dipramit Majumdar} \address{Somnath Jha, School of Physcial Sciences, Jawaharlal Nehru University, New Delhi - 110067, India}\email{jhasomnath@gmail.com} \address{Dipramit Majumdar, Indian Institute of Science Education and Research Pune, Pashan, Pune 411008, Maharashtra, India} \email{dipramit@gmail.com}

\maketitle

\begin{abstract} {

We establish a duality result proving the `functional equation' of the characteristic ideal of the Selmer group associated to a nearly ordinary Hilbert modular form over the cyclotomic $\Z_{p}$ extension of a totally real number field. 
Further, we use this result to establish a duality or algebraic `functional equation'   for the `big' Selmer groups associated to the corresponding nearly ordinary Hida deformation.  The multivariable cyclotomic Iwasawa main conjecture for nearly ordinary Hida family of Hilbert modular forms is not established  yet and this can be thought of as an evidence to the validity of this Iwasawa main conjecture. 
We  also prove a functional equation for the `big'  Selmer group associated to an ordinary Hida family of elliptic modular forms over the $\Z_{p}^{2}$ extension of an  imaginary quadratic field. 

}\end{abstract}

\tableofcontents
\section*{Introduction}
 We fix an odd rational prime $p$ and $N$ a natural number prime to $p$. Throughout, we fix an embedding $\iota_\infty$ of a fixed algebraic closure $\bar{\Q}$ of $\Q$  into $\C$ and also an embedding $\iota_l$ of $\bar{\Q}$ into a fixed algebraic closure ${\bar{\Q}}_\ell$ of the  field $\Q_\ell$ of the $\ell$-adic numbers, for every prime $\ell$. Let $F$ denote a totally real number field and $K$ denote an  imaginary quadratic field. For any number field $L$, $S_{L}$ will denote a finite set of places of $L$ containing the primes dividing $Np$. The  cyclotomic $\Z_{p}$ extension of $L$ will be denoted by $L_{\cyc}$ and   the unique $\Z_{p}^{2}$ extension of $K$ will be denoted by $K_\infty$. Set $\Gamma := \text{Gal}(L_{\cyc}/L) \cong \Z_{p}$ and $\Gamma_K := \text{Gal}(K_{\infty}/K) \cong \Z_{p}^{2}$. Let $B$ be a  commutative, complete, noetherian, normal, local ring of characteristic $0$ with finite residue field of characteristic  $p$. We will denote by $B[[ \Gamma ]]$ (resp. $B[[ \Gamma_{K}]]$) the Iwasawa algebra of $\Gamma$ (resp. $\Gamma_{K}$) with coefficient in $B$. Let $M$ be a finitely generated torsion $B[[\Gamma]]$ (resp. $B[[\Gamma_{K}]]$) module. Then $M^{\iota}$ denote the $B[[\Gamma]]$ (resp. $B[[\Gamma_{K}]]$) module whose underlying abelian group is the same as $M$ but the $\Gamma$ (resp. $\Gamma_{K}$) action is changed via the involution sending $\gamma \mapsto \gamma^{-1}$ for every $\gamma \in \Gamma$ (resp. $\gamma \in \Gamma_{K}$). We denote the characteristic ideal of $M$ in $B[[\Gamma]]$ (resp. $B[[\Gamma_{K}]]$)) by $Ch_{B[[\Gamma]]} (M)$ (resp. $Ch_{B[[\Gamma_{K}]]} (M)$). The main results of the article is the following theorem.

\begin{theorem*}[Theorem \ref{mnm} and Theorem \ref{mnmm}]\label{THMB}
 Let $\mathcal{R}$ be a branch of Hida's universal nearly ordinary Hecke algebra associated to nearly ordinary Hilbert modular forms (resp. Hida's universal ordinary Hecke algebra associated to ordinary elliptic modular forms). Let $\T_\mathcal R$ be a $G_{F}$ (resp. $G_{\Q}$) invariant lattice associated to `big' Galois representation $(\rho_{\mathcal{R}}, \mathcal{V}_{\mathcal{R}})$ and set  $\T_\RR^{\ast} : = \text{Hom}_{\RR}(\T_\RR, \RR(1))$. Let $\mathcal{X}(\T_\mathcal R/F_{\cyc})$ (resp.  $\mathcal{X}(\T_\mathcal R/K_{\infty})$) denote the dual Selmer group of $\mathcal{R}$ over $F_{\cyc}$ (resp. $K_{\infty}$). Then under certain conditions, as ideals in $\mathcal{R}[[\Gamma]]$ (resp. $\mathcal{R}[[\Gamma_{K}]]$), we have the equality  
$$Ch_{\mathcal{R}[[\Gamma]]}(\mathcal{X}(\T_\mathcal R/F_{\cyc})) = Ch_{\mathcal{R}[[\Gamma]]}(\mathcal{X}(\T_\mathcal R^{\ast}/F_{\cyc})^{\iota})$$
 (resp.
  $$Ch_{\mathcal{R}[[\Gamma_{K}]]}(\mathcal{X}(\T_\mathcal R/K_{\infty})) = Ch_{\mathcal{R}[[\Gamma_{K}]]}(\mathcal{X}(\T^{\ast}_\mathcal R/K_{\infty})^{\iota}) .) \qed$$ 
 \end{theorem*}

 
Theorem \ref{mnm} and Theorem \ref{mnmm} generalizes \cite[Theorem 5.2]{jp}. An ingredient in the  proof of Theorem \ref{mnm} is Theorem \ref{fefibre}, where we prove a functional equation for the Selmer group of a single Hilbert modular form. Theorem \ref{fefibre} is a variant of \cite[Theorem 2]{gr} and \cite[Theorem 4.2.1]{pr}. 
The motivation for  the above theorems  is explained below. 

Let $E$ be an elliptic curve defined over $\Q$ with good ordinary reduction at an odd prime $p$. Let $L_E$ be the complex $L$ function of $E$. Then by modularity theorem, $E$ is modular and $L_E$ coincides with the $L$-function of a weight $2$ newform of level $N_E$, where $N_E$ is the conductor of $E$. Moreover, $L_E$  has analytic continuation to all of $\C$ and  if we set  $$\Lambda_E(s): = N_E^{s/2}(2\pi)^{-s}\Gamma(s)L_E(s)$$ to be the completed $L$-function of $E$, where $\Gamma(s)$ is the usual $\Gamma$ function, then (due to Hecke) 
\begin{equation}\label{mot1}
\La_E(2-s) = \pm \La_E(s), s \in \C.
\end{equation}  Thus $L_E$ satisfies a functional equation connecting values at $s$ and $2-s$, where $s \in \C$. Let $\phi$ vary over the Dirichlet characters of  $\Gamma =\text{Gal}(\Q_\cyc/\Q)$ i.e. $\phi \in \hat{\Gamma}$, then the twisted $L$-function $L_{E}(s,\phi)$, will also satisfy a functional equation similar to \eqref{mot1} connecting the values of $L_{E}(2-s,\phi)$ and $L_{E}(s,\phi^{-1})$. In particular, $L_{E}(1,\phi)$ and $L_{E}(1,\phi^{-1})$ are related by a functional equation.

Now, by the work of Mazur and Swinnerton-Dyer \cite{msd}, the $p$-adic $L$-function of $E$ exists. Let $g_E(T)$ be the power series representation of the $p$-adic $L$-function of $E$ in $\Z_p[[\Gamma]]\otimes \Q_p \cong \Z_p[[T]] \otimes \Q_p$. Then  we have, 
 \begin{itemize}
 \item  $g_E(0) = (1-\alpha_p^{-1})^2L_E(1)/{\Omega_E}$,
 \item  $g_E(\phi(T) - 1 ) = \frac{L_E(1,\phi) {\beta_p}^n}{\tau(\phi)\Omega_E}$, for a  character $\phi$ of $\Gamma$ of order $p^n \geq 1$. Here $\alpha_p +\beta_p = a_p$, $\alpha_p \beta_p =p$ with $p  \nmid  \alpha_p$, $\Omega_E$ is the real period of $E$ and  $\tau(\phi)$ is the Gauss sum of $\phi$.
 \end{itemize}
We assume that $g_E \in \Z_p[[\Gamma]]$. Then using the above interpolation properties, we can deduce from \eqref{mot1}, a  functional equation for the $p$-adic $L$-function (also proven  in \cite{mtt}), given by  
\begin{equation}\label{mot2}
 g_E(T)  = u_Eg_E(\frac{1}{1+T} -1),
 \end{equation}
where $u_E$ is a unit in the ring $\Z_p[[\Gamma]] \cong \Z_p[[T]]$. In other words, we have an equality of ideals in $\Z_p[[\Gamma]]$,

\begin{equation}\label{mot3}
 (g_E(T))  = (g_E(\frac{1}{1+T} -1)).
 \end{equation}
 
 By the  cyclotomic Iwasawa main conjecture for $E$ (cf. \cite[$\S$ 3.5]{es}), the $p$-adic  $L$-function $g_E(T) $ is a characteristic ideal of the (dual) $p^\infty$-Selmer group  $X(T_pE/\Q_\cyc)$. Thus, we should get 
 \begin{equation}\label{mot4}
 Ch_{\Z_p[[\Gamma]]}(X(T_pE/{\Q_\cyc})) = Ch_{\Z_p[[\Gamma]]}(X(T_pE^*/{\Q_\cyc})^\iota),
 \end{equation}
   as ideals in the Iwasawa algebra $\Z_p[[\Gamma]]$. (Notice that by Weil pairing, $T_pE^* \cong T_pE$). 
   
   Indeed, under certain assumptions, \eqref{mot4} is a corollary of the main conjecture of Iwasawa theory for elliptic curves (proven in  \cite[Corollary 3.34]{es}, also see \cite{ka}) together with the fact that the $p$-adic $L$-function $g_E$ satisfies  functional equation \eqref{mot3}. However, for any elliptic curve $E$ defined over $\Q$ which is ordinary at $p$, \eqref{mot4} was already proven  by \cite[Theorem 2]{gr} (and also independently in \cite[Theorem 4.2.1]{pr}) purely algebraically and without assuming the existence of $g_E$. The proofs uses  duality  and pairing in cohomology, (like the Poitou-Tate duality, generalized Cassels-Tate pairing of Flach) among other tools.

Now for any compatible system of $l$-adic representations associated to a motive, a complex $L$-function is  defined and  one can think of similar questions.   For example, for a  normalized cuspidal Hilbert eigenform $f$, which is nearly ordinary at primes $\mathfrak p$ of a totally real number field $F$ dividing the prime  $ p$,  one can associate a compatible system of $l$-adic representation. Furthermore, one can  define a Selmer group $X(T_f/{F_\cyc})$ using the  Galois representation of the Galois group $G_F:=\Gal (\bar{F}/F)$. Under suitable conditions (for example, non-critical slope), the $p$-adic $L$-function for $f$, which interpolates the complex $L$-function, exists (see \cite{di}, also see \cite{ba}); and a precise  Iwasawa main conjecture for $f$ over $F_\cyc$ can also be formulated (for example, see \cite{xw}). However, this   cyclotomic Iwasawa main conjecture for $f$  is not proven yet. In Theorem \ref{fefibre}, we prove the functional equation for the characteristic ideal of the Selmer group of $f$ i.e.

\begin{equation}\label{mot5}
 Ch_{O_f[[\Gamma]]}(X(T_f/{F_\cyc})) = Ch_{O_f[[\Gamma]]}(X(T^*_f/{F_\cyc})^\iota),  
 \end{equation}
algebraically (without assuming the existence of the $p$-adic $L$-function or the Iwasawa main conjecture of $f$). Thus, Theorem \ref{fefibre} can be thought of as a modest evidence towards the validity for the cyclotomic Iwasawa main conjecture for $f$.

   Now, let us consider the nearly ordinary Hida deformation of Hilbert modular forms.  A   `several variable' $p$-adic $L$-function  (say $\mathcal L^p$) associated to (a branch $\RR$ of the) nearly ordinary Hida family ${\mathbf H}_{\mathcal N, O}$ over $F_\cyc$   will interpolate  the special values of the complex $L$-functions of the various individual nearly ordinary normalized cuspidal  Hilbert  eigenforms lying in the family (cf.  \cite{o4}, \cite{di}). Hence,  $\mathcal L^p$  should also satisfy a functional equation. Again, by the `several variable'   Iwasawa  main conjecture over $F_\cyc$ for a nearly  ordinary Hida family ${\mathbf H}_{\mathcal N, O}$ of Hilbert modular forms (cf. \cite{xw}), $\mathcal L_{p}$ should be a characteristic ideal of the `big' Selmer group  $\mathcal X(\mathcal T_\RR/F_\cyc)$ of the branch $\RR$ of the nearly ordinary Hida family. Thus, we would expect to  get a  `functional equation' stating  
   \begin{equation}\label{mot6}
   Ch_{R[[\Gamma]]}(\mathcal X(\mathcal T_\RR/F_\cyc))= Ch_{R[[\Gamma]]}(\mathcal X(\T_\RR^*/F_\cyc)^\iota)
   \end{equation} where $\Gamma = \text{Gal}(F_\cyc/F)$.  Again,   we will prove this fact algebraically (without any assumption on the analytic side) in Theorem \ref{mnm} and thus, in turn,   this can be thought of as a modest evidence for the validity of the Iwasawa main conjecture for the nearly ordinary Hida deformation of Hilbert modular forms over $F_\cyc$.

   From an purely algebraic point of view of duality and pairing in cohomology, Theorem \ref{mnm}  can also be thought as a direct generalization of  result of \cite[Theorem 2]{gr} in the nearly ordinary Hida deformation setting.

 
 An entirely  parallel argument  as above,  in the setting of ordinary Hida deformation of   elliptic modular forms over the $\Z_p^2$ extension of an imaginary quadratic field $K$,   works  as the motivation for Theorem \ref{mnmm}. (Functional equation for elliptic modular forms over the cyclotomic $\Z_p$ extension of $K$ was discussed in \cite[Theorem 5.2]{jp}.) In this case though, we would like to stress that  for an imaginary quadratic field $K$,  under certain hypotheses,   the three variable   Iwasawa main conjecture over $K_\infty$ for an ordinary Hida family of elliptic modular form has been  proven in \cite[$\S$ 3.6.3]{es}. Indeed,  in  that article   a suitable three variable $p$-adic $L$-function  for the ordinary Hida family has been constructed (\cite[$\S$ 3.4.5]{es}). It is known  that this 3 variable $p$-adic $L$-function  should satisfy a functional equation.     Thus, at least  in principle, the work of \cite{es}  should also establish the equality  
 $$Ch_{\mathcal R[[\Gamma_K]]}(\mathcal X(T_\RR/K_\infty)) = Ch_{\mathcal R[[\Gamma_K]]}(\mathcal X(T^*_\RR/K_\infty)^\iota) $$ 
 in $\mathcal R[[\Gamma_K]]$.  However,   we would like to mention that  our proof of the functional equation for the Selmer group $\mathcal X(\T_\mathcal R/K_\infty)$ (Theorem \ref{mnmm}) is simple  and we do  not need to make use of the vast tools  involved in the proof of  the 3 variable main conjecture of \cite{es}.  Moreover,  in Theorem \ref{mnmm}, we do not need some of the hypotheses of the proof of the main conjecture (see \cite[$\S$ 3.6]{es}).

  The key idea of the proof of  Theorem \ref{fefibre} is to use generalized Cassels-Tate pairing of Flach along with a ``control theorem'' (Theorem \ref{ctcyc}).  The central idea of the proof of Theorem \ref{mnm} (and Theorem \ref{mnmm}) can be explained in three steps. First, we show that for infinitely many arithmetic points the specialization map is a pseudo-isomorphism. Secondly, we use the fact that functional equation holds at the fibre  for infinitely many arithmetic specialization. Finally, we use some suitable lifting techniques, generalization of results of  \cite{o3}, to obtain  our results. This gives a simple proof of the desired functional equation of the `big' Selmer group of the nearly ordinary Hida family. 

 
 The structure of the article is as follows. In section 1, we discuss some preliminary results in two parts. In subsection \ref{sec1}, we discuss preliminaries related to the Hida deformation for nearly ordinary Hilbert modular forms and ordinary elliptic modular forms, only to the extent  which we need in this article.  In  subsection \ref{sec2}, we define various Selmer group involved.  In section \ref{sec3}, we prove a control theorem for Hilbert modular form and deduce Theorem \ref{fefibre}. In section  \ref{sec4}, we discuss the specialization results connecting the `big' Selmer groups with the Selmer groups of the individual Hilbert modular forms at the fibres and prove the main theorem in the Hilbert modular form  case (Theorem \ref{mnm}). We prove  the second version of the main theorem for elliptic modular forms over $\Z_p^2$ extension  in section \ref{sec6}  (Theorem  \ref{mnmm}).

 \n {\bf Acknowledgment} :
We would like to thank T. Ochiai, A. Raghuram,  B. Baskar, R. Greenberg, A. Pal and R. Sujatha for discussions. During this project, first named author was initially supported by JSPS postdoctoral fellowship and later by DST Inspire Faculty Award grant. The second author was initially supported by ISI  postdoctoral  fellowship and later by IISER Pune postdoctoral fellowship. 
\section{Preliminaries}
\subsection{Preliminary results on Hilbert Modular Forms}\label{sec1}
\subsubsection{Hilbert Modular Forms and Nearly Ordinary Hecke Algebra}
In this subsection, we collect some basic results about  nearly ordinary Hida deformation  of Hilbert modular forms and ordinary Hida deformation of elliptic modular forms, which are needed in the course of this article. All the results in this section are  well known and can be found in the literature (cf.  \cite{h1}, \cite{h2}, \cite{wi}, \cite{wi2}). Our presentation of results in this subsection, in many cases, follows the presentation of \cite{fo}.

 Let $p$ be an odd prime. Let $F$ be a totally real number field of degree $d$, $\OO_{F}$ be the ring of integers of $F$, and $J_{F}$ denotes the set of embedding of $F$ into $\R$. To an ideal $\M$ of $\OO_{F}$, we attach standard compact open subgroups $K_{0},K_{1}$ and $K_{11}$ of $Gl_{2}(\OO_{F}\otimes_{\Z} \hat{\Z})$ as follows:
$$K_{0}(\M)= \Big\{ \begin{pmatrix}
                              a & b\\
                               c & d
                              \end{pmatrix} \in Gl_{2}(\OO_{F}\otimes_{\Z} \hat{\Z}) |  c \equiv 0 ~(mod  ~ \M)\Big\}$$

$$K_{1}(\M)= \Big\{ \begin{pmatrix}
                              a & b\\
                               c & d
                              \end{pmatrix} \in K_{0}(\M) |  d \equiv 1 ~(mod ~ \M)\Big\}$$
                              
$$K_{11}(\M)= \Big\{ \begin{pmatrix}
                              a & b\\
                               c & d
                              \end{pmatrix} \in K_{0}(\M) |  a,d \equiv 1 ~(mod ~ \M)\Big\}.$$

\begin{defn}
A weight $k= \sum_{\tau \in J_{F}} k_{\tau} \tau$ is an element of $\Z[J_{F}]$, an arithmetic weight is a weight such that $k_{\tau} \geq 2$ for all $\tau \in J_{F}$ and $k_{\tau}$ has constant parity. A parallel weight is an integral multiple of the weight $t = \sum_{\tau \in J_{F}} \tau$. Two weights are said to be equivalent if their deference is a parallel weight. To an arithmetic weight $k$, one associates a weight $v \in \Z[J_{F}]$, called the parallel defect of $k$, which satisfies $k+2v \in \Z t$.
\end{defn}

Let $\OO$ be the ring of integers of a finite extension of $\Q_{p} $ which contains all conjugates of $F$. For $k$ an arithmetic weight and $v$ its parallel defect, $S_{k,w}(U;\OO)$ denotes the holomorphic cusp forms of weight $(k,w)$ of level $U$ and coefficient in $\OO$, where $U$ is a finite index subgroup of  $Gl_{2}(\OO_{F}\otimes_{\Z} \hat{\Z})$ containing $K_{11}(\M)$ for some $\M \subset \OO_{F}$, and $w =k+v-t$. A cuspidal Hilbert modular form  $S_{k,w}(U;\OO)$ is called primitive if it is not a Hilbert modular form of weight $(k,w)$ and of level smaller than $U$. A normalized primitive Hilbert eigenform is called a Hilbert newform. To $f \in S_{k,w}(U;\OO)$, one can naturally associate an automorphic representation $\pi_{f} \in Gl_{2}(\mathbb{A}_{F})$. 

Fix an ideal $\mathcal{N}$ of $F$ which is prime to p and for any $s \in \N$, we have an action of $G = (\OO_{F}\otimes_{\Z} \Z_{p})^{\ast} \times ((\OO_{F}\otimes_{\Z} \Z_{p})^{\ast}/\bar{\OO_{F}^{\ast}})$ on $S_{k,w}(K_{1}(\mathcal{N}) \cap K_{11}(p^{s});\OO)$. We have an action of the $p$-Hecke operator $T_{0}(p)$, normalized according to the parallel defect $v$, on the space $S_{k,w}(K_{1}(\mathcal{N}) \cap K_{11}(p^{s});\OO)$. The largest $\OO$ submodule of $S_{k,w}(K_{1}(\mathcal{N}) \cap K_{11}(p^{s});\OO)$ on which $T_{0}(p)$ acts invertibly is denoted by $S_{k,w}^{n.o.}(K_{1}(\mathcal{N}) \cap K_{11}(p^{s});\OO)$. A form $f \in S_{k,w}(K_{1}(\mathcal{N}) \cap K_{11}(p^{s});\OO)$ is called nearly ordinary if $f \in S_{k,w}^{n.o.}(K_{1}(\mathcal{N}) \cap K_{11}(p^{s});\OO)$.

The nearly ordinary Hecke algebra ${\bf H}_{k,w}(K_{1}(\mathcal{N}) \cap K_{11}(p^{s});\OO)$ of weight $(k,w)$ and level $K_{1}(\mathcal{N}) \cap K_{11}(p^{s})$ is defined to be the $\OO$ subalgebra of  $End_{\OO}(S_{k,w}^{n.o.}(K_{1}(\mathcal{N}) \cap K_{11}(p^{s});\OO))$ generated by the Hecke operators. The $\OO$ algebra ${\bf H}_{k,w}(K_{1}(\mathcal{N}) \cap K_{11}(p^{s});\OO)$ is finite flat over $\OO$.

Let $k$ be an arithmetic weight and $v$ be its parallel defect. By the perfect duality between ${\bf H}_{k,w}(K_{1}(\mathcal{N}) \cap K_{11}(p^{s});\OO)$ and $S_{k,w}^{n.o.}(K_{1}(\mathcal{N}) \cap K_{11}(p^{s});\OO)$, giving an eigen cuspform $f \in S_{k,w}^{n.o.}(K_{1}(\mathcal{N}) \cap K_{11}(p^{s});\OO)$ is equivalent to giving an algebra homomorphism 
$$q_{f}: {\bf H}_{k,w}(K_{1}(\mathcal{N}) \cap K_{11}(p^{s});\OO) \surj {\bf H}_{k,w}(K_{1}(\mathcal{N}p^{s}) ;\OO) \to \bar{\Q}_{p}$$
sending $T \in {\bf H}_{k,w}(K_{1}(\mathcal{N}p^{s}) ;\OO)$ to $a_{1}(f|T)$.

Let $\Lambda_{\OO}$ denote the completed group algebra $\OO[G/G_{tors}]$. The algebra $\Lambda_{\OO}$ is non-canonically isomorphic to the power series algebra $\OO[[X_{1},\cdots,X_{r}]]$, where $r=1+d+\delta_{F,p}$, $\delta_{F,p}$ be the defect of the Leopoldt's conjecture for $F$ at $p$.

Let the nearly ordinary Hecke algebra ${\bf H}_{\mathcal{N},\OO}$ be the inverse limit w.r.t. $s$ of the ${\bf H}_{2t,0}(K_{1}(\mathcal{N}) \cap K_{11}(p^{s});\OO)$. By fundamental work of Hida, for any arithmetic weight $k \in \Z[J_{F}]$, we have $ {\bf H}_{\mathcal{N},\OO} \cong \underset{s}\varprojlim ~{\bf H}_{k,w}(K_{1}(\mathcal{N}) \cap K_{11}(p^{s});\OO)$. The nearly ordinary Hecke algebra ${\bf H}_{\mathcal{N},\OO}$ is finite torsion free $\Lambda_{\OO}$ module and hence a semi-local ring. Let $\mathfrak{a}$ be one of the finitely many ideals of height zero in ${\bf H}_{\mathcal{N},\OO}$. The algebra $\RR={\bf H}_{\mathcal{N},\OO}/\mathfrak{a}$ is called a branch of ${\bf H}_{\mathcal{N},\OO}$.

\begin{defn}\label{kerxi}
For a weight $k \in \Z[J_{F}]$, an algebraic character $\xi: G=(\OO_{F}\otimes_{\Z} \Z_{p})^{\ast} \times ((\OO_{F}\otimes_{\Z} \Z_{p})^{\ast}/\bar{\OO_{F}^{\ast}}) \to \bar{\Q}_{p}^{\ast}$ of weight $(k,w)$ is a character of the form, $\xi(a,z) = \psi(a,z)\chi_{\cyc}^{[n+2v]}(z)a^{n}$, where $\psi$ is a character of finite order, $[n+2v]$ is the unique integer satisfying $n+2v = [n+2v]t$ (Recall, $n=k+2t$, $w= k+v-t$).  An algebraic character  of weight $(k,w)$ is an arithmetic character  of weight $(k,w)$ if its restriction to $\OO_{F}^{\ast} \subset (\OO_{F}\otimes_{\Z} \Z_{p})^{\ast}$ is trivial. An algebra homomorphism, $\xi \in \Hom(\Lambda_{\OO},\bar{\Q}_{p})$ is algebraic (resp. arithmetic)
of weight $(k,w)$ if $\xi|_{G}$ is algebraic (resp. arithmetic) of weight $(k,w)$. If $R$ is a finite $\Lambda_{\OO}$ algebra, $\xi \in \Hom(R,\bar{\Q}_{p})$ is algebraic (resp. arithmetic) of weight $(k,w)$ if $\xi|_{\Lambda_{\OO}}$ is algebraic (resp. arithmetic) of weight $(k,w)$.\\
A prime ideal $P_{\xi} \subset R$ which is defined as the kernel of an algebraic (resp. arithmetic) specialization of $\xi$ of $R$ is called an algebraic (resp. arithmetic) point.
\end{defn}

For any $k$ be an arithmetic weight and any of its parallel defect $v$, and for any nearly ordinary eigen cuspform $f \in S_{k,w}^{n.o.}(K_{1}(\mathcal{N}) \cap K_{11}(p^{s});\OO)$ which is new at every prime diving $\mathcal{N}$, there exists a unique branch $\RR $ of ${\bf H}_{\mathcal{N},\OO}$  and a unique arithmetic specialization $\xi_f : \RR \to \bar{\Q}_{p}$ of weight $(k,w)$ such that $\xi_f (\RR)$ is canonically identified with $q_{f}({\bf H}_{k,w}(K_{1}(\mathcal{N}) \cap K_{11}(p^{s});\OO))$.\\
Let $\RR$ be a branch of ${\bf H}_{\mathcal{N},\OO}$. Then for any $k$ be an arithmetic weight and any of its parallel defect $v$, and for any arithmetic specialization $\xi:\RR \to \bar{\Q}_{p}$ of weight $(k,w)$, there exists a unique nearly ordinary eigen cuspform $f_{\xi} \in {\bf H}_{k,w}(K_{1}(\mathcal{N}p^{s});\OO)$ for some $s$, such that $\xi(\RR)$ is canonically identified with $q_{f_{\xi}}({\bf H}_{k,w}(K_{1}(\mathcal{N}) \cap K_{11}(p^{s});\OO))$.  We will donate the set of arithmetic points of $\RR$ by $\mathfrak X (\RR)$. For each $\xi \in \mathfrak X (\RR)$,  $ P_\xi = \text{ker}(\xi)$ is a coheight 1 prime ideal in $\RR$.
\subsubsection{Galois Representation}

Galois representation associated to a Hilbert modular eigen cuspform was constructed and studied by Carayol,  Ohta, Wiles, Taylor and Blasius-Rogawski. We briefly recall their results in the following two theorems.

\begin{theorem}\label{1}
Let $f \in S_{k,w}(K_{1}(\M); \bar{\Q}_{p})$ be a normalized eigen cuspform of arithmetic weight $k$, and let $K$ be a finite extension of $\Q_{p}$ containing all Hecke eigenvalues for $f$. Then there exists a continuous irreducible $G_{F}$ representation $V_{f} \cong K^{\oplus 2}$, which is unramified outside $\M p$ and satisfies
$$ det(1- \text{Fr}_{\lambda} X|V_{f}) = 1- T_{\lambda} (f) X + S_{\lambda}(f) X^2 $$
for all $\lambda \nmid \M p $, where $T_{\lambda}$ (resp. $S_{\lambda}$) is the Hecke operator induced by the coset class $K_{1}(\M) \begin{pmatrix} 
                                          1 & 0 \\
                                          0 & \pi_{\lambda}
                                          \end{pmatrix} K_{1}(\M)$ (resp. $K_{1}(\M) \begin{pmatrix} 
                                          \pi_{\lambda} & 0 \\
                                          0 & \pi_{\lambda}
                                          \end{pmatrix} K_{1}(\M)$), where $\pi_{\lambda}$ is a uniformizer at $\lambda$ and $Fr_{\lambda}$ is the geometric Frobenius at $\lambda$.
 
 The $G_{F}$ representation $V_{f}$ is known to be irreducible, and hence characterized upto isomorphism by the above equation.                                        
\end{theorem}

\begin{rem}
Let $f \in S_{k,w}(K_{1}(\M); \bar{\Q}_{p})$ be a normalized eigen cuspform of arithmetic weight $k$, and let $K$ be a finite extension of $\Q_{p}$ containing all Hecke eigenvalues for $f$ as in Theorem \ref{1}. If $\p \mid p$, let $c(\p,f)$ be the $T(\p)$ eigenvalue of $f$. We say $f$ is ordinary at $\p$ if $c(\p,f)$ is a unit in the ring of integers of $K$ and $f$ is ordinary at $p$ if and only if for all $\p \mid p$, $f$ is ordinary at $\p$.
\end{rem}
 Next theorem describes the local properties of the  Galois representation $V_{f}$.

\begin{theorem}\label{2}
Let $f \in S_{k,w}(K_{1}(\M); \bar{\Q}_{p})$ be a normalized eigen cuspform. Let $w_{max} =\text{max} ~\{w_{\tau}, \tau \in J_{F} \}$. Let $V_{f}$ (resp. $\pi_{f}$) be the Galois representation (resp. automorphic representation) associated to $f$.
\begin{enumerate}
\item If $\lambda \nmid p$, then 
 \begin{enumerate}
  \item The inertia group $I_{\lambda}$ at $\lambda$ acts on $V_{f}$ through infinite quotient iff $\pi_{f,\lambda}$ is a Steinberg representation. In this case, $V_{f}$ has a unique filtration by graded pieces of dimension one:
  $$ 0 \to (V_{f})_{\lambda}^{+} \to V_{f} \to (V_{f})_{\lambda}^{-} \to 0 $$
 which is stable under the decomposition group $G_{\lambda}$ at $\lambda$. The inertia group $I_{\lambda}$ acts on $(V_{f})_{\lambda}^{+}$ (resp. $(V_{f})_{\lambda}^{-}$) through a finite quotient of $I_{\lambda}$. An eigenvalue $\alpha$ of the action of a lift of $Fr_{\lambda}$ to $G_{\lambda}$ on $(V_{f})_{\lambda}^{+}$ (resp. $(V_{f})_{\lambda}^{-}$) is an algebraic number satisfying $|\alpha|_{\infty} = (N_{F/\Q}(\lambda))^{\frac{w_{max}+1}{2}}$ (resp. $(N_{F/\Q}(\lambda))^{\frac{w_{max}-1}{2}}$). 
 \item If $I_{\lambda}$ acts on $V_{f}$ through a finite quotient, the action of $I_{\lambda}$ is reducible iff $\pi_{f,\lambda}$ is principal series. If $I_{\lambda}$ acts on $V_{f}$ through a finite quotient, an eigenvalue $\alpha$ of the action of a lift of $Fr_{\lambda}$ to $G_{\lambda}$ on $V_{f}$ is an algebraic number satisfying $|\alpha|_{\infty} = (N_{F/\Q}(\lambda))^{\frac{w_{max}}{2}}$.
 \end{enumerate}
\item  If $\p |p$, and  if $f$ is nearly ordinary at $\p$. Then $V_{f}$ has a unique filtration by graded pieces of dimension one:
  $$ 0 \to (V_{f})_{\p}^{+} \to V_{f} \to (V_{f})_{\p}^{-} \to 0 $$
which is stable under the decomposition group $G_{\p}$ at $\p$, and Hodge-Tate weight of $(V_{f})_{\p}^{+}$ is greater then Hodge-Tate weight of $(V_{f})_{\p}^{-}$.
\end{enumerate}
\end{theorem}

\begin{rem}
Let $f \in S_{k,w}(K_{1}(\M); \bar{\Q}_{p})$ be a normalized eigen cuspform of arithmetic weight $k$ with associated Galois representation $V_{f}$ over $K$ as in Theorem \ref{1}. If $f$ is $p$-ordinary, then for all primes $\p \mid p$,
$$V_{f}|_{G_{\p}} \sim \begin{bmatrix}
                                         \epsilon_{\p}  & \ast \\ 
                                         0  & \delta_{\p}
                                         \end{bmatrix},$$
 where $\epsilon_{\p},\delta_{\p}$ are characters of $G_{\p}$ with values in $K^{\ast}$ and $\delta_{\p}$ is unramified. In the case of nearly ordinary $f$ at $\p$, our Galois representation restricted $G_{\p}$ looks same, except $\delta_{\p}$ need not be unramified.
\end{rem}
The following two theorems are the Hida family versions of the two above theorem, and are due to the work of Wiles and Hida.
\begin{theorem}\label{3}
Let $\RR$ be a branch of ${\bf H}_{\mathcal{N}, \OO}$. Then there exists a finitely generated torsion free $\RR$ module $\T_{\RR}$ with continuous $G_{F}$ action, which satisfies the following properties:
\begin{enumerate}
\item The vector space $\mathcal{V}_{\RR}:= \T_{\RR} \otimes_{\RR} \mathcal{K}$ is of dimension $2$ over $\mathcal{K}$, where $\mathcal{K}$ is the fraction field of $\RR$.
\item The representation $\rho_{\mathcal{R}}$ of $G_{F}$ on $\mathcal{V}_{\RR}$ is irreducible and unramified outside primes dividing $\mathcal{N}p\infty$.
\item  For any arithmetic weight $(k,w)$, and for any nearly ordinary eigen cuspform $f \in S_{k,w}^{n.o.}(K_{1}(\mathcal{N}) \cap K_{11}(p^{s});\OO)$ which corresponds to the arithmetic weight $\xi = \xi_f$ on the branch $\RR$, $T_{f} = \T_{\RR} \otimes_{\RR} \xi_f (\RR)$ is a lattice of the Galois representation $V_{f}$ associated to $f$.
\end{enumerate}
\end{theorem}

Next theorem characterizes the local behavior of the Galois representation associated to $\RR$.
\begin{theorem}\label{4}
Let $\RR$ be a branch of ${\bf H}_{\mathcal{N}, \OO}$ and $\mathcal{V}_{\RR}=\mathcal{V}$ be the Galois representation over $\mathcal{K}$. Then,
\begin{enumerate}
\item For every prime $\lambda \nmid \mathcal{N}p$, we have:
$$ det(1-\text{Fr}_{\lambda} X|\mathcal{V}) = 1- T_{\lambda} X + S_{\lambda} X^2 $$
where $T_{\lambda}$ and $S_{\lambda}$ are Hecke operators on $\RR$ at $\lambda$ which is obtained at the limit of the Hecke operators in Theorem \ref{1} at finite levels.
\item For every prime $\p |p$, we have a canonical filtration obtained as the limit of the filtration given in Theorem \ref{2}:
$$ 0 \to \mathcal{V}_{\p}^{+} \to \mathcal{V} \to \mathcal{V}_{\p}^{-} \to 0$$
 which is stable under the action of the decomposition group $G_{\p}$ at $\p$.
\end{enumerate}
\end{theorem}

For an ideal $\mathcal{N}$ in $\OO_{F}$ which is prime to $p$, we have a Hida's nearly ordinary Hecke algebra ${\bf H}_{\mathcal{N},\OO}$. We fix a branch $\RR$ of ${\bf H}_{\mathcal{N},\OO}$ and a representation $\T=\T_{\RR}$ as in Theorem \ref{3}. We assume that we have a $G_{\p}$ stable filtration 
$$0 \to \T_{\p}^{+} \to \T \to \T_{\p}^{-} \to 0 $$
by finite type $\RR$ modules $\T_{\p}^{+}$ and $\T_{\p}^{-}$ with continuous $G_{\p}$ action which gives the exact sequence
$$ 0 \to \mathcal{V}_{\p}^{+} \to \mathcal{V} \to \mathcal{V}_{\p}^{-} \to 0$$
in Theorem \ref{4} by taking base extension to $\mathcal{K}$.

Let $\mathfrak{m}_{\mathcal{R}}$ denote the maximal ideal of $\RR$ and $\mathbb F_\RR$ be the finite field $\RR/{\mathfrak m_\RR}$. Associated  to the Galois representation  of $G_F$ in Theorem \ref{3}, there exist a canonical residual Galois representation $\bar{\rho}_{\mathcal{R}} :G_F \lra GL_2(\mathbb F_\RR)$. Throughout this article, we assume the following two hypotheses on  this $\bar{\rho}$.
\begin{hypothesis}\label{irr}
{\bf (Irr)}: The residual representation $\bar{\rho}_{\mathcal{R}}$ of $G_F$   is absolutely irreducible.
\end{hypothesis}

\begin{hypothesis}\label{Dist}
{\bf (Dist)}:  As before, let $G_{\p}$ be the decomposition subgroup of $G_F$ at $\p$. The restriction  of the residual representation at the decomposition subgroup i.e. $ \bar{\rho}_{\mathcal{R}} \mid_{G_{\p}}$ is an extension of two distinct characters of $G_{\p}$ with values in $\mathbb F_\RR^{\ast}$ for each $\p|p$.
\end{hypothesis}

\begin{rem}
Conditions {\bf (Irr)} and {\bf (Dist)} together implies that the representation $\T$ can be chosen to be free of rank two over $\mathcal{R}$ and, for each $\p \mid p$, the graded pieces $\T_{\p}^{+}$ and $\T_{\p}^{-}$ are both free of rank one over $\mathcal{R}$.
\end{rem}

\subsubsection{ordinary Hida deformation of elliptic modular forms}

 We summarize the relevant results for $p$-ordinary Hida family of  elliptic modular forms in the following remark. This details can be found in the  literature (also  in \cite[Section 1]{jp}).

\begin{rem}\label{elipticcase}
Let $f= \sum a_{n} q^{n}$ be a normalized elliptic eigenform of weight $k \geq 2$, level $N$ and nebentypus $\psi$. We say that $f$ is $p$-ordinary if $\iota_{p}(a_{p})$ is a $p$-adic unit. Also assume that $f$ is $p$-stabilized ($p$-refined). By the work of Eichler-Shimura, Deligne, Mazur-Wiles, Wiles and many other people, to  such an $f$, one can associate a Galois representation $\rho_{f}: G_{\Q} \to GL(V_{f})$, here $V_{f}$ is a two dimensional vector space over a finite extension of $\Q_{p}$, which is unramified outside $Np$, and for any prime $l \nmid Np$, arithmetic $\text{Frob}_{l}$ has the characteristic polynomial $X^{2}-a_{l}X+\psi(l)l^{k-1}$, moreover, restricted to the decomposition subgroup at $p$, we have, $V_{f}|_{G_{p}} \sim \begin{bmatrix} \epsilon_{1} & \ast \\  0 & \epsilon_{2} \end{bmatrix}$ with $\epsilon_{2}$ unramified. We have a notion of  Hida family and arithmetic points for elliptic modular forms of tame level $N$. A theorem of Hida asserts that if $(N,p)=1$ and $f$ is a $p$-stabilized newform of weight $k \geq 2$, tame level $N$, then there exists a branch  $\mathcal{R}$ of an ordinary Hida family and an arithmetic point $\xi: \mathcal{R} \lra \bar{\Q}_p$ such that $\xi(\RR)$ canonically  corresponds  with $f$. By the work of Hida and Wiles, to $\mathcal{R}$ one can associate a big Galois representation of dimension two, $\rho_{\mathcal{R}} : G_{\Q} \to GL(\mathcal{V_{\mathcal{R}}})$, where $\mathcal{V_{\mathcal{R}}}$ is a vector space of dimension two over $\mathcal{K}$, the fraction field of $\mathcal{R}$. The Galois representation $\rho_{\mathcal{R}}$ is unramified at all finite places outside $Np$ and for a prime $l \nmid Np$, one has $det(1-\text{Fr}_{l} X|_{\mathcal{V}}) = 1 - T_{l} X + S_{l} X^{2}$, where $T_{l},S_{l}$ are Hecke operators on $\mathcal{R}$ at $l$, moreover, $\mathcal{V}_{\mathcal{R}}|_{G_{p}} \sim \begin{bmatrix} \tilde{\epsilon}_{1} & \ast \\  0 & \tilde{\epsilon}_{2} \end{bmatrix}$ with $\tilde{\epsilon}_{2}$ unramified. 
Similar to the  case of nearly ordinary Hilbert modular forms, throughout the article, we make the following two hypotheses.

\begin{hypothesis}\label{irr2}
{\bf (Irr)}:The residual representation $\bar{\rho}_\RR$ of $G_\Q$   is absolutely irreducible.
\end{hypothesis}

\begin{hypothesis}\label{Dist2}
{\bf (Dist)}:For the representation $ \bar{\rho}_\RR \mid_{G_{p}}$, we have  $\tilde{\epsilon}_{1} \neq  \tilde{\epsilon}_{2}\pmod{\mathfrak m_\RR}$.
\end{hypothesis}

Under {\bf (Irr)} and {\bf (Dist)}, we have the lattices $T_{f}$ in $V_{f}$ and $\mathcal{T}_{\mathcal{R}}$ in $\mathcal{V}_{\mathcal{R}}$ invariant under $\rho_{f}$ and $\rho_{\mathcal{R}}$ respectively, such that $T_{f} \cong  \mathcal{T}_{\mathcal{R}} \otimes_{\mathcal{R}} \xi(\mathcal{R})$. Moreover,  $\mathcal{T}_{\mathcal{R}}$ has a filtration as $G_{p}$ module,
$$0 \to \T_{\mathcal{R}}^{+} \to \T_{\mathcal{R}} \to \T_{\mathcal{R}}^{-} \to 0 $$
such that the graded pieces $\T_{\mathcal{R}}^{+}$ and $\T_{\mathcal{R}}^{-}$ are free $\mathcal{R}$ module of rank one.
\end{rem}

\subsection{ Various Selmer groups}\label{sec2}
 We fix a totally real number field $F$ and as $S = S_F$ is  a finite set of finite places of $F$ containing the primes lying above $\mathcal{N}p$.
Let $f \in S_{k,w}^{n.o.}(K_{1}(\mathcal{N}) \cap K_{11}(p^{s});\OO)$. Let $V_{f}$ denote the Galois representation associated to $f$ over $K_{f}$, a finite extension of $\Q_{p}$ containing all Hecke eigenvalues of $f$ as in Theorem \ref{1}. We denote the ring of integer of $K_{f}$ by $O_{f}$. Let $T_f \subset V_f$ be a $G_{F}$ invariant lattice. Thus we have an induced action of $G_{F}$ on the discrete module $A_f:= V_f/T_f $. Further, since $\rho_f$ is nearly ordinary at $p$, $A_f$ has a filtration as a $G_{\p}$ module for every $\p |p$,
\begin{equation}\label{iaff}
 0 \lra (A_f)_{\p}^{+} \lra A_f \lra (A_f)_{\p}^{-} \lra 0,
 \end{equation}
 where  both ${(A_f)_{\p}^{+}}^\vee$ and  ${(A_f)_{\p}^{-}}^\vee$ are free over $O_{f}$ of rank 1. 
 
 \n Let  $\mathcal L$ be an finite or infinite  algebraic extension of $F$ and $w$ will denote a prime in $\mathcal L$. The notation $w \mid S $ will mean the prime $w$ of $\mathcal L$ is lying above a prime in  $S$.  Given  such a prime $w$, let $G_{ w}$ and $I_w$ respectively  denote  the decomposition subgroup and inertia subgroup for the extension $\bar{\Q}/\mathcal L$ with respect to the primes  $\bar{w}/w$, where we have fixed a prime $\bar{w}$ of $\bar{\Q}$ over $w$. We denote the  Frobenius element at $w$ by $\text{Fr}_w$ so that $G_{w}/{I_w} \cong <\text{Fr}_w>.$

 Let $\RR$ be a branch of ${\bf H}_{\mathcal{N}, \OO}$. By Theorem \ref{3}, we have a free $\RR$ lattice $\T=\T_{\RR}$. Define,
 $$\A= \mathcal{A}_{\RR} := \T \otimes_{\RR}\text{Hom}_{\text{cont}}(\RR, \Q_p/{\Z_p}).$$
 For any arithmetic character $\xi$ of $\RR$, we have from definition \ref{kerxi},  $\mathcal A[P_\xi] \cong A_{f_\xi}. $ Using Theorem \ref{4} and  by   the assumption \ref{Dist}, we get $\mathcal A$ has a filtration as a $G_{\p}$ module
\begin{equation}\label{bgp}
 0 \lra \mathcal A_{\p}^{+} \lra \mathcal A \lra \mathcal A_{\p}^{-} \lra 0,
 \end{equation}
where  both ${\mathcal A_{\p}^{+}}^\vee$ and  ${\mathcal A_{\p}^{-}}^\vee$ are free $\RR$ modules of  rank 1. Also we have $\mathcal A_{\p}^-[P_\xi] \cong (A_{f_\xi})_{\p}^{-}$. 

 Let $L$ be a finite extension of $K$.  We define various Selmer groups associated to  $f$ and $\RR$, defined over $L$.
 \begin{defn}[Greenberg Selmer group of $f$] \label{self}
\begin{equation}
 S(A_f/ L)  =  \mathrm{ker}(H^1(F_S/L, A_f) \lra \underset {w \mid S, w \nmid p}{\oplus} H^1(I_w, A_f)^{G_{w}}\underset {w \mid \p \mid p}{\oplus} H^1(I_w, (A_f)_{\p}^{-})^{G_{w}})
  \end{equation}
 \end{defn}

\begin{defn}[Strict Selmer group of $f$]
\begin{equation}
S^{'}(A_{f}/L) := \mathrm{ker} (H^1(F_S/L, A_f) \lra \underset {w \mid S, w \nmid p}{\oplus} H^1(I_w, A_f)^{G_w}\underset {w \mid \p  \mid p}{\oplus} H^1(G_w, (A_f)_{\p}^{-}))
\end{equation}
\end{defn}

 \begin{defn}[Greenberg Selmer group of $\RR$] \label{bself}
\begin{equation} 
 S(\mathcal A/ L)  =  \mathrm{ker}(H^1(F_S/L, \mathcal A) \lra \underset {w \mid S, w \nmid p}{\oplus} H^1(I_w, \mathcal A)^{G_{w}}\underset {w \mid \p \mid p}{\oplus} H^1(I_w, \mathcal A_{\p}^{-})^{G_{w}}) 
  \end{equation}
 \end{defn}

\begin{defn}[Strict Selmer group of $\RR$]\label{bsself}
\begin{equation}
S^{'}(\mathcal A/L) := \mathrm{ker} (H^1(F_S/L, \mathcal A) \lra \underset {w \mid S, w \nmid p}{\oplus} H^1(I_w, \mathcal A)^{G_w}\underset {w \mid \p  \mid p}{\oplus} H^1(G_w, \mathcal A_{\p}^{-}))
\end{equation}
\end{defn}

\begin{defn}\label{pont0}
 Let $S^\perp \in \{S, S'\}$.    Define the Pontryagin dual
\begin{equation}\label{pont1}
X^\perp(T_f/L) = \text{Hom}_{\text{cont}}(S^\perp(A_f/L), \Q_p/{\Z_p})
\end{equation}
where $X^\perp = X$ if $ S^\perp =S $ and $X^ \perp = X'$ if $S^\perp = S'$. \\
Define $ \mathcal X^\perp(\T_\RR/L)$ by replacing $X^\perp $ by $\mathcal X^\perp $, $A_f $ by $\mathcal A$ and $T_f $ with $\T_\RR$ in  \eqref{pont1}.
\end{defn}

For an infinite extension $F_\infty$ of $F$, the Selmer group $S^\perp(A_f/F_\infty)$ (respectively    $S^\perp(\mathcal A/F_\infty)$) is defined   by taking the inductive  limit  of $S^\perp(A_f/{L'}) $ (resp. $S^\perp(\mathcal A/L')$) over all finite extensions $L'$ of $F$ contained in $F_\infty$ with respect to the natural restriction maps.  The corresponding Pontryagin duals are denoted by $X(T_f/F_\infty)$ and $ \mathcal X(\T_\RR/F_\infty)$ respectively.

\n  Under the natural action of $\Gamma = \text{Gal}(F_\cyc/F)$,  $X(T_f/F_\cyc)$ (respectively $\mathcal X(\T_\RR/F_\cyc)$) acquires the structure of a  $O_f[[\Gamma]]$ (respectively $\RR[[\Gamma]]$) module. Also note that for $B \in \{ A_f, \mathcal A \} $, we have  $ 0 \lra S'(B/F_\cyc) \hookrightarrow S(B/F_\cyc)$.


 
Next we  discuss various types of twisted Selmer groups. For any $\Z_p$ module $M$, let $M(1)$ denotes the Tate twist of $M$ by the $p$-adic cyclotomic character $\chi_p: \Gamma \lra \Z_p^\times$.  Define $$\T_\RR^{\ast} : = \text{Hom}_{\RR}(\T_\RR, \RR(1)).$$  We have a  corresponding filtration of $\T^{\ast}$ as a $G_{\p}$ module  $$ 0 \lra ({\mathcal T_\RR^{\ast}})_{\p}^{+} \lra \mathcal T_\RR^{\ast} \lra ({\mathcal T_\RR^{\ast}})_{\p}^{-} \lra 0,  $$ where  the graded pieces are defined as $ ({\mathcal T_\RR^{\ast}})_{\p}^{+} : = \text{Hom}_{\RR}(({\T_\RR})_{\p}^{-}, \RR(1))$ and  $({\mathcal T_\RR^{\ast}})_{\p}^{-} : = \text{Hom}_{\RR}(({\T_\RR})_{\p}^{+}, \RR(1))$. 
 We can now define 
$$\mathcal A^{\ast} : = \mathcal T_\RR^{\ast} \otimes_{\RR}\text{Hom}_{\text{cont}}(\RR, \Q_p/{\Z_p}).$$ From the above discussion, we can get a filtration of $\A^{\ast} $  as in \eqref{bgp}.

\n Also, corresponding to  a newform $f$,  we define $$T_f^{\ast} : = \text{Hom}_{O_f}(T_f, O_f(1)).$$ Then it is easy to see that the quotient $\mathcal T^{\ast}\otimes_{\RR} \xi(\RR)$ is isomorphic to $T^{\ast}_{f_\xi}. $ Also define $A_f^{\ast} = T_f^{\ast} \otimes \Q_p/\Z_p$. Then as in \eqref{iaff}, there is a filtration $$ 0 \lra ({A^{\ast}_f})_{\p}^{+} \lra A^{\ast}_f \lra ({A^{\ast}_f})_{\p}^{-} \lra 0,  $$ with ${({A^{\ast}_f})_{\p}^{+}}^\vee$ and ${({A^{\ast}_f})_{\p}^{-}}^\vee$  are free   $O_f$ module of rank 1. 

From the above discussions,  by making obvious modifications in  the definitions \ref{self}, \ref{bself} and \ref{pont0}, we can now define the Greenberg Selmer groups $S(A_f^{\ast}/\mathcal L)$, $S(\mathcal A^*/\mathcal L) $ and their  respective Pontryagin duals $X(T^{\ast}_f/\mathcal L)$ and $\mathcal X(T^*_\RR/\mathcal L)$ for any finite or infinite extension $\mathcal L$ of $F$.

Let $\rho : \Gamma \lra O_f^\times$ be a character.  Set $T_\rho = O_f(\rho)$, the $G_F$ module with underlying group  $O_f$ and  an $ G_F$ action on it via $\rho$. Set $T_f(\rho) = T_\rho \otimes_{O_f}T_f $, $A_f(\rho) = T_\rho \otimes_{O_f} A_f$ and  $(A_f)_{\p}^{-}(\rho) = T_\rho \otimes_{O_f} (A_f)_{\p}^{-}$ with the diagonal action of $G_F$. Let $M$ be  an $O_f[[\Gamma]]$ module.  Define $M(\rho) = T_\rho \otimes_{O_f} M$ with $\gamma \in \Gamma$ acting by diagonal action. Applying $\otimes_{O_f} T_\rho$ to the filtration in \eqref{iaff}, we get a filtration for  $T_\rho \otimes A_f$. Proceeding  in a way similar to the definition \ref{self}, define the  Greenberg Selmer groups with respect to the `twist' $\rho$, $S(A_f \otimes T_\rho/\mathcal L)$ and  $X(T_\rho \otimes T_f/\mathcal L)$, for any  extension $\mathcal L$  of $F$ (possibly infinite).   As $\rho$  acts trivially on $G_{F_\cyc}$,  we notice  that  
\begin{equation}\label{twist}
X(T_\rho \otimes T_f/F_\cyc) \cong X(T_f/F_\cyc) \otimes T_{\rho^{-1}}.
 \end{equation}
 In particular $X(T_\rho \otimes T_f/F_\cyc)$ is a finitely generated torsion  $O_f[[\Gamma]]$  module if and only if $X(T_f/F_\cyc)$ is so.

\begin{rem}\label{twsel}
Let  $f$ be a newform  nearly ordinary at $p$. Then we can express $T^*_f \cong T_{f^*}\otimes O_f(\chi_p^t)$  as Galois modules for some $t \in \Z$ and for some newform $f^*$ which is nearly ordinary at $p$, has the same level and weight as $f$  but  possibly different character.  Hence for any extension $\mathcal L$ of $F$,  we deduce that $X(T_\rho \otimes T_f^*/\mathcal L)  \cong X(T_{\rho'} \otimes T_{f^*}/\mathcal L)$  for certain character   $\rho' : \Gamma \lra O_f^\times$. In particular,   $X( T_f^*/F_\cyc)  \cong X(T_{f^*}/F_\cyc) \otimes O_f(s)$ for some $s \in \Z$.
\end{rem}

\begin{rem}{Selmer group for $p$-ordinary  elliptic modular forms and for the corresponding Hida family:}\label{seleliptic}

Let $f \in S_k(\Gamma_0(Np^r), \psi)$ be a $p$-ordinary, $p$-stabilized (elliptic) newform. Then by remark \ref{elipticcase}, via its Galois representation, we can associate  to $f$, a lattice $T_f$. Set $A_f = V_f/T_f$. Assume the conditions {\bf (Irr)} and {\bf (Dist)}.  Then, again by remark \ref{elipticcase}, to a branch $\RR$ of a $p$-ordinary Hida family of elliptic modular forms  of tame level $N$, we have a free $\RR$ lattice $\T_\RR$ and we  can define $\A =  \mathcal{A}_{\RR} := \T \otimes_{\RR}\text{Hom}_{\text{cont}}(\RR, \Q_p/{\Z_p}).$  Then by $p$-ordinarity, both $A_f$ and $\mathcal A$ are equipped  with canonical filtration as  $G_p$ modules.

Let $K$ be an imaginary quadratic field and $K_\cyc$ and $K_\infty$ be respectively, the unique cyclotomic and $\Z_p^2$ extension of $K$. We assume that $p$ splits in $K$ and the discriminant $D_K$ is coprime to $N$. Let $S = S_K$ be a finite set of places of $K$ containing the primes dividing $Np$. Then proceeding in a similar way as in definitions \ref{self} - \ref{pont0}, we can define the Greenberg Selmer groups and the strict Selmer group of $f$ and $\RR$ over $K_\cyc$ and $K_\infty$. In fact, we will use the same symbols  used in the above definitions. However, there is no case of confusion, as we deal with the elliptic modular forms and their ordinary Hida family only in section \ref{sec6}.

\end{rem}


\section{functional equation for a Hilbert modular form}\label{sec3}

A control theorem is a widely used  tool  in Iwasawa theory. We prove  a `control theorem'  for   the twisted Selmer group $X(T_f \otimes T_\rho /F_\cyc)$ with $\rho$ as before. A control theorem in case of  elliptic modular forms  was discussed in \cite[Theorem 3.1]{jp}. Recall, there is a tower  of fields  $F= F_0 \subset  ... \subset F_n \subset ...  \subset F_\cyc$ are such that $Gal(F_n/F) \cong \Z/{p^n\Z}$. Set $\Gamma_n = \text{Gal}(F_\cyc/F_n)$. Given a cuspidal Hilbert newform  $f \in S_{k,w}^{n.o}(K_{1}(N) \cap K_{11}(p^s);\mathcal{O})$, let $\rho_f$ and $\pi_f$ respectively be the associated  Galois representation and automorphic representation.  Let $f$ be  nearly ordinary at every prime $\mathfrak{p}$ in $F$, $\mathfrak p \mid p$.  Then for any such $\mathfrak p $ dividing $p$, the Frobenius $\text{Fr}_{\mathfrak p}$ acts on the 1 dimensional  subspace $V^{-}_{f, \mathfrak p}$ with the eigenvalue  $\alpha_{f, \mathfrak p}$ (say). 
\begin{defn}\label{exceptional}
Let $f$ be as above. Define $f$ to be \it{exceptional} if  $|\alpha_{f,\mathfrak p}|_{\C} =1 $ for some $\mathfrak p \mid p$.
\end{defn}

\begin{rem}\label{grc}
From the local Langlands correspondence for Hilbert modular forms due to Carayol \cite{ca} and generalized Ramanujan Conjecture, which is known for Hilbert modular form due to Blassias \cite{bl};  it follows that the condition of $f$ being  exceptional i.e. $|\alpha_{f,\mathfrak p}|_{\C} =1 $ for some $\mathfrak p \mid p$ happens  if  for some $\mathfrak p \mid p$, the $\mathfrak p$ component $\pi_{f,\mathfrak p}$ of the automorphic representation $\pi_f$ is Steinberg or its twist.
\end{rem}

\begin{theorem}\label{ctcyc} 
Let $f$ be a Hilbert newform in $S_{k,w}^{n.o}(K_{1}(N) \cap K_{11}(p^s);\mathcal{O})$ as above.
Assume $f$ is  not exceptional. Let $\rho$ be a character $\rho : \Gamma \lra O_f^\times$ as above. Then the kernel and the cokernel of the map $$X(T_f \otimes T_\rho /F_\cyc)_{\Gamma_n} \lra X(T_f \otimes T_\rho/F_n) $$ are finite groups for all $n$ with their cardinality  uniformly bounded independent of $n$. 
\end{theorem} 
 
 \proof  Let $v_n$  be a prime of $F_n$ lying above  $S$   and let $v_c$  be a  prime of $F_\cyc$ lying above it. Given such a prime $v_n$ we fix a prime $\bar{v}$ in $\bar{\Q}$ lying above it. Recall from $\S$ \ref{sec2},  $G_{v_n}$ (resp. $G_{v_c}$) denotes the decomposition subgroup $\bar{\Q}/F_n$ (resp. $\bar{\Q}/F_\cyc$) for the prime $\bar{v}/{v_n}$ (resp.  $\bar{v}/v_\infty$). The corresponding inertia subgroups are  $I_{v_n}$ and  $I_{v_c}$ respectively. Similarly, $G_{v_{c}/v_{n}}$  (resp. $I_{v_c/v_n}$) denotes the decomposition subgroup (resp. inertia subgroup) of the Galois group of  $F_\cyc/F_n$ with respect to the primes $v_{\infty}/v_{n}$. Also the various Frobenius elements are given by  $<\text{Fr}_{\mathfrak{p}}>  : = {\frac{G_{\mathfrak{p}}}{I_{\mathfrak{p}}}}$,  $<\text{Fr}_{v_{n}}>  : = {\frac{G_{v_{n}}}{I_{v_{n}}}}$ and  $ <\text{Fr}_{v_c/v_n}>: =\frac{G_{v_c/v_n}}{I_{v_c/v_n}}$.  Set  
 \[
 J_{v_n}  =
  \begin{cases}
    H^1(I_{v_n}, T_\rho \otimes A_f)  & \text{if } v_n \mid S, v_n \nmid p\\
    H^1(I_{v_n}, T_\rho \otimes (A_f)_{\p}^{-})^{G_{v_n}} & \text{if } v_n \mid \p  \mid p,
    \end{cases}
\]
\[
  J_{v_c}  =
  \begin{cases}
    H^1(I_{v_c}, T_\rho \otimes A_f)  & \text{if } v_c \in S, v_c \nmid p\\
    H^1(I_{v_c}, T_\rho \otimes (A_f)_{\p}^{-})^{G_{v_c}} & \text{if } v_c  \mid \p  \mid p,
    \end{cases}
\]
We study the commutative diagram
\begin{equation}\label{controlcyc} 
\xymatrix{  0 \ar[r] & S( T_\rho \otimes A_f/F_\cyc)^{\Gamma_n} \ar[r]  & H^1(F_S/F_\cyc , T_\rho \otimes A_f)^{\Gamma_n}  \ar[r] & ( \underset {v_{c} \mid S} {\oplus} J_{v_c})^{\Gamma_n} \\ 
                   0 \ar[r] & S( T_\rho \otimes A_f/F_n) \ar[r] \ar[u]_{\alpha_n}&  H^1(F_S/F_n , T_\rho \otimes  A_f) \ar [r]  \ar[u]_{\phi_n}&   \underset {v_n \mid S} {\oplus}J_{v_n} \ar[u]_{\theta_n = \oplus \theta_{v_n}}}
\end{equation}
First, we prove that $\text{ker}(\alpha_n) $ is finite for all $n$. We will show $\text{ker}(\phi_n)$ is finite for all $n$.  Note that $\text{ker}(\phi_n) \cong H^1(\Gamma_n, (T_\rho \otimes A_f)^{G_{F_\cyc}})$. As $\Gamma_n$ is topologically cyclic and $T_\rho \otimes A_f$ is a cofinitely generated $\Z_p$ module is follows that  $\text{ker}(\phi_n)$ is finite if and only if $H^0(\Gamma_n, (T_\rho \otimes A_f)^{G_{F_\cyc}})$ is finite. Also, to show $(T_\rho \otimes A_f)^{G_{F_n}}$ is finite, it suffices to show $V_f(\rho)^{G_{F_n}} = 0.$ If $V_f(\rho)^{G_{F_n}} \neq 0$, then $V_f(\rho)$ contains a trivial $G_{F_n} $ sub-representation. In that case, we choose a place $\lambda \nmid Np $ and then restrict $G_F$ representation   to $\text{Fr}_\lambda$. Then    considering the eigenvalues of $\text{Fr}_\lambda$, we immediately get a contradiction by Theorem \ref{2}(1)(a), Theorem \ref{2}(1)(b). 

\noindent As $H^0(F_n, V_f(\rho)) =0$ for every $n$ and $\text{ker}(\alpha_n) $ is finite for all $n$, we deduce from \cite[Theorem 3.5(i)]{o2} that $\text{ker}(\alpha_n) $ is uniformly bounded independent of $n$. 

For  $\text{coker}(\alpha_n)$, first  note that $\text{coker}(\phi_n) \subset H^2(\Gamma_n, (T_\rho \otimes A_f)^{G_{F_\cyc}}) = 0 $ as $p$-cohomological dimension  of $\Gamma_n =$ 1 for any $n$. Using the Snake lemma, it suffices  to show that the kernel of $\theta_n$ are finite and  uniformly bounded independent of $n$. Now, there are only finite number of primes in $F_\cyc$ lying above a given prime in any $F_n$. Hence  it is enough to prove  $\text{ker}(\theta_{v_n})$ is finite for each $v_n \mid S$.
Now for a prime $v_n \mid S $ such that $v_n \nmid p$, we have 
 \begin{equation} \label{1001}
 \text{ker}(H^1(I_{v_n}, T_\rho \otimes  A_f) \lra H^1(I_{v_c}, T_\rho \otimes A_f)) \cong H^1(I_{v_n}/I_{v_c}, T_\rho \otimes  A_f^{I_{v_c}}).
 \end{equation}
 The last isomorphism follows from the inflation-restriction sequence of Galois cohomology.  We know that the  cyclotomic $\Z_p$ extension of any number field is unramified outside primes above $p$. Thus  $I_{v_n}/I_{v_c}  \cong I_{v_c/v_n} = 0 $ whenever $v_n \nmid p$. Using \eqref{1001} in diagram \eqref{controlcyc}, it is immediate that $\text{ker}(\theta_{v_n})$ vanishes for $v_n \nmid p$.

To consider $\text{ker}(\theta_{v_{n}})$ for primes $v_n \mid \p \mid p$, we study 
\begin{align}\label{lo2}
\text{ker}(H^1(I_{v_{n}}, T_\rho \otimes  (A_f)_{\p}^{-})^{ G_{v_{n}}} \lra H^1(I_{v_{c}},T_\rho \otimes  (A_f)_{\p}^{-})^{G_{v_{c}}}) \nonumber \\ \cong H^1(I_{v_{c}/v_{n}}, (T_\rho \otimes (A_f)_{\p}^{-})^{G_{v_{c}}})^{<\text{Fr}_{v_c/v_n}>}
\end{align}
As,  $\rho$ acts trivially on $G_{v_{c}}$,  $$H^1(I_{v_{c}/v_{n}}, (T_\rho \otimes (A_f)_{\p}^{-})^{G_{v_{c}}}) \cong H^1(I_{v_{c}/v_{n}}, T_\rho \otimes ((A_f)_{\p}^{-})^{G_{v_{c}}}). $$   Now $F_\cyc/F$ is totally ramified at $\mathfrak{p}$ for all $\mathfrak{p}|p$. Hence, $I_{v_{c}/v_{n}} \cong \Z_p$ for every $v_n$ lying above $\p \mid p$. Note that, as an abstract group $T_\rho \cong O_f $, $ ((A_f)_{\p}^{-})^\vee \cong O_f$.  Hence \\$H^1(I_{v_{c}/v_{n}}, T_\rho\otimes {(A_f)_{\p}^{-}}^{G_{v_{c}}}) = 0 $ unless $I_{v_{c}/v_{n}}$ acts trivially on $T_\rho \otimes ({(A_f)_{\p}^{-}})^{G_{v_{c}}}$. 

Thus it suffices to consider the case where $I_{v_{c}/v_{n}}$  is acting trivially on $T_\rho \otimes ({(A_f)_{\p}^{-}})^{G_{v_{c}}}$, as otherwise kernel in \eqref{lo2} is $0$.  Then 
$$H^1(I_{v_{c}/v_{n}}, T_\rho \otimes {(A_f)_{\p}^{-}}^{G_{v_{c}}}) \cong \text{Hom}(I_{v_{c}/v_{n}}, (T_\rho \otimes {(A_f)_{\p}^{-}})^{G_{v_{c}}}). $$ 
 
\n As $F_\cyc/F$ is abelian, the action of the Frobenius  $ <\text{Fr}_{v_{c}/v_{n}}>$ on   $I_{v_{c}/v_{n}}$ (via lifting and conjugation) is trivial. Hence,  the module in \eqref{lo2} is isomorphic to 
\begin{align*} 
 \text{Hom}_{<\text{Fr}_{v_{c}/v_{n}}>}(I_{v_{c}/v_{n}}, (T_\rho \otimes {(A_f)_{\p}^{-}})^{G_{v_{c}}})  \\ \cong \text{Hom}(I_{v_{c}/v_{n}}, (T_\rho \otimes (A_f)_{\p}^{-})^{G_{v_{n}}}) \\ \cong \text{Hom}(I_{v_{c}/v_{n}}, (T_\rho \otimes (A_f)_{\p}^{-})^{\frac{G_{v_{n}}}{I_{v_{n}}}})\end{align*}
 
We claim that $(T_\rho \otimes (A_f)_{\p}^{-})^{<\text{Fr}_{v_{n}}>}$ is finite.  On the one dimensional vector space corresponding to  $(A_f)_{\p}^{-}$, $\text{Fr}_{\mathfrak{p}}$ acts by multiplication  by  $\alpha_{\mathfrak{p}}(f)$. By our assumption in this theorem that $f$ is not  exceptional and remark \ref{grc}, it follows that the eigenvalue of $\text{Fr}_{\mathfrak{p}}$ acting on the 1 dimensional line  corresponding to $(A_f)_{\p}^{-}$  is not a root of unity. On the other hand, $\rho(g) =1$ for any $g  \in G_{F_\cyc}$ and  $F_\cyc/F$ is totally ramified at any prime in $F$ lying above $p$. Hence  the eigenvalue corresponding to the action of $\text{Fr}_{\mathfrak{p}}$ on $T_\rho$ is a root of unity. Combining these  facts, we deduce that  the  eigenvalue of  $\text{Fr}_{\mathfrak{p}}$ on the line corresponding to $T_\rho \otimes (A_f)_{\p}^{-}$ is not a root of unity. Hence $\text{Fr}_{v_{n}}$  acts non-trivially on  $T_\rho \otimes (A_f)_{\p}^{-}$ for any $n$ and any  $\rho$ as before. Hence ${(T_\rho \otimes (A_f)_{\p}^{-})}^{<\text{Fr}_{v_{n}}>}$ is indeed finite.  Also as  $<\text{Fr}_{v_{n}}>  ~ \cong  ~<\text{Fr}_{v_{c}}>$ for $n \gg 0$, we deduce that   the size of ${(T_\rho \otimes (A_f)_{\p}^{-})}^{<\text{Fr}_{v_{n}}>}$ is independent of $n$ for large enough $n$. As $I_{v_{c}/v_{n}} \cong \Z_p$ for every $n$;  for any  $\rho$ and every $n \geq 0$, the module in  \eqref{lo2}, given by   $\text{Hom}(I_{v_{c}/v_{n}}, (T_\rho \otimes (A_f)_{\p}^{-})^{<\text{Fr}_{v_{n}}>})$,  is also finite  with its cardinality bounded  independent of $n$.  Hence same is true for $\text{ker}(\theta_{v_{n}})$ for  $v_n \mid \p$. This completes the proof.
\qed 
 
Let $O$ be the ring of integers of a finite extension of $\Q_p$. Take  $M$ to be a finitely generated $\Lambda$ module, where $\Lambda = O[[\Gamma]]$. Let us denote $ \text{Ext}^{~1}_\La(M, \La)$ by $a_\La^1(M)$.

\begin{lemma}\label{a1}
 Let $M$ be a finitely generated torsion $\La = O[[\Gamma]]$ module such that $M_{\Gamma_n}$ is finite for each $n$. Then 
 $$ a_\La^1(M) \cong \underset{n}\varprojlim ~({M^\iota}^\vee)^{\Gamma_n}.$$ 
\end{lemma}
\proof  See \cite[\S 1.3, page 733]{pr}.  \qed

\begin{lemma}\label{pa1}
Let $M$ be a finitely generated torsion $\La = O[[\Gamma]]$ module. Then   $Ch_{\La }(M) = Ch_{\La }(a^1_\La(M))$, considered as ideals in $O[[\Gamma]].$
\end{lemma}
\proof See  \cite[Lemma 3.5]{jp}).
\qed

\begin{lemma}\label{tw}
Let $M$  be a finitely generated torsion $\La =O[[\Gamma]]$ module. Then there exists a character $\rho : \Gamma= \text{Gal}(F_\cyc/F) \lra  \text{Aut}(O) $ such that $(M(\rho))_{\Gamma_n}$  is finite for every $n$, where $M(\rho)$  is as defined in section \ref{sec2}. 
\end{lemma}

\begin{proof}
This is well known, for example see \cite[\S 2.6, Page 740]{pr}.
\end{proof}

\begin{lemma}\label{fgsn}
 For any cuspidal Hilbert newform $f$, the dual Selmer groups  $ X(T_f/F_\cyc)$  and  $ X(T_f^*/F_\cyc)$   are finitely generated  $O_f[[\Gamma]]$  modules.
 \end{lemma}\label{2fgs}
\proof The proof is similar to the one in elliptic modular form case (see for example \cite[Lemma 3.7]{jp}). 

\medskip

\noindent Throughout the rest of the article we make this assumption - 

\begin{hypothesis}\label{torf}
{\bf $(\text{Tor})$} = {\bf (Tor$_f$)} : For any cuspidal Hilbert newform $f$, $X(T_f/F_\cyc)$ is finitely generated torsion $O_f[[\Gamma]]$ module.
\end{hypothesis}

\begin{corollary}\label{4555}
It follows from remark \ref{twsel} that by hypothesis ({\bf Tor}), we have for  any cuspidal Hilbert newform $f$, $X(T^*_f/F_\cyc)$ is also   torsion over $O_f[[\Gamma]]$.
\end{corollary}

\begin{theorem}\label{fefibre}
Let the notation be as before.  Let $f \in S_{k,w}^{n.o}(K_{1}(N) \cap K_{11}(p^s);\mathcal{O})$ be a Hilbert newform nearly ordinary at $\mathfrak{p}|p$, which is not exceptional (as defined in theorem \ref{ctcyc}).  Assume  {\bf $(\text{Tor})$}  holds. Then the functional equation holds for $X(T_f/F_\cyc)$ i.e. we have the equality of  ideals in $O_f[[\Gamma]]$, $$Ch_{O_f[[\Gamma]]}(X(T_f/F_\cyc)) =Ch_{O_f[[\Gamma]]}( X(T^*_f/F_\cyc)^\iota) .$$

\end{theorem}

\proof 
By the assumption {\bf ($\text{Tor}$)} and corollary \ref{4555}, both $X(T_f/F_\cyc)$ and $X(T^*_f/F_\cyc)$ are torsion over $O_f[[\Gamma]]$. Thus   we can find a $\rho$ by lemma  \ref{tw} such that $X(T_f\otimes T_\rho/F_n)$ and $X(T^*_f\otimes T_{\rho^{-1}}/F_n)$ are  both finite  groups for every $n$. Then from the generalized Cassels-Tate pairing of Flach (see \cite[ 3.1.1]{pr}), we obtain that $S(T_{\rho^{-1}} \otimes A_f^*/F_n) \cong X(T_{\rho} \otimes T_f/F_n)$ for every $n$.   Hence we get
\begin{equation}\label{xf}
X(T_{\rho} \otimes T_f/F_\cyc) \cong \underset{n}\varprojlim ~ X(T_{\rho} \otimes T_f/F_n)  \cong \underset{n}\varprojlim ~S(T_{\rho^{-1}} \otimes A^*_f/F_n). 
\end{equation}
By  Theorem \ref{ctcyc} and remark  \ref{twsel}, we see that  the kernel  and the cokernel of the natural restriction map, given by $S(T_{\rho^{-1}} \otimes A^*_f/F_n) \stackrel {\alpha^*_n}{\lra} S(T_{\rho^{-1}} \otimes A^*_f/F_\cyc)^{\Gamma_n} $, are finite groups and their size  uniformly bounded independent of $n$. Thus we obtain from \eqref{xf} that the  induced map 
\begin{equation}\label{phi}
X(T_{\rho} \otimes T_f/F_\cyc) \stackrel{\phi_{\rho}}{\lra} \underset{n}\varprojlim ~S(T_{\rho^{-1}} \otimes A^*_f/F_\cyc)^{\Gamma_n}
\end{equation} is a  $O_f[[\Gamma]]$  pseudo-isomorphism.
We have  $$ \underset{n}\varprojlim ~S(T_{\rho^{-1}} \otimes A^*_f/F_\cyc)^{\Gamma_n} = \underset{n}\varprojlim ~{(X(T_{\rho^{-1}} \otimes T^*_f/F_\cyc)^\vee)}^{\Gamma_n} \stackrel {\text{Lemma} ~ \ref{a1}} \cong a^1_\La(X(T_{\rho^{-1}} \otimes T^*_f/F_\cyc)^\iota).$$ Combining this with \eqref{phi} we get  an  $O_f[[\Gamma]]$ module pseudo-isomorphism
\begin{equation}\label{thetarho}
X(T_f\otimes T_\rho/F_\cyc)  \stackrel{\theta_\rho} \lra a^{1}_{\La}(X(T^*_f\otimes T_{\rho^{-1}}/F_\cyc)^\iota). 
\end{equation}
We recall from \eqref{twist},  $X(T_f\otimes T_\rho/F_\cyc) \cong X(T_f/F_\cyc) \otimes T_{\rho^{-1}}$. On the other hand, we have $a^1_\La(X(T^*_f\otimes T_{\rho^{-1}}/F_\cyc)^\iota) \cong a^1_\La(X(T^*_f/F_\cyc)^\iota)\otimes T_{\rho^{-1}}$ (see \cite[Page 744, \S3.2.1-3.2.2]{pr}). Thus tensoring  \eqref{thetarho} with $T_{\rho}$, we get a pseudo-isomorphism of $O_f[[\Gamma]]$  modules $$X(T_f/F_\cyc)   \lra a^{1}_{\La}(X(T^*_f/F_\cyc)^\iota) $$  which is independent of $\rho$. Hence $Ch_{O_f[[\Gamma]]}(X(T_f/F_\cyc)) = Ch_{O_f[[\Gamma]]}(a^1_\La(X(T^*_f/F_\cyc)^\iota ))  $ as  ideals in $O_f[[\Gamma]]$. By applying Lemma \ref{pa1}, we get that $$ Ch_{O_f[[\Gamma]]}(X(T_f/F_\cyc))= Ch_{O_f[[\Gamma]]}(X(T^*_f/F_\cyc)^\iota ). \hfill \quad \quad \quad \quad  \qed$$


\section{functional equation for a nearly ordinary Hida family }\label{sec4}

We begin by proving a specialization result relating the `big' Selmer group with the individual Selmer groups.

\begin{theorem}\label{spl}
Assume {\bf (Irr)}, {\bf (Dist)}. Let $\mathcal{R}$ be a branch of ${\bf H}_{\mathcal{N,\OO}}$ and assume that $\mathcal{R} $ is a power series ring in many variable (i.e.  we assume that $\RR  \cong O[[X_{1},\cdots,X_{r}]]$, where $r=d + 1+ \delta_{F,p}$, here $\delta_{F,p}$ is the defect of Leopoldt's conjecture for $F$ at $p$, $d = [F:\Q]$, and $O$ is the ring of integer of some finite extension of $\Q_p$). Let  $s_\xi^\vee$ be the natural specialization map
\begin{equation}\label{splz}
\x(\T_\RR/F_\cyc)/{P_\xi   \x(\T_\RR/F_\cyc)} \stackrel{s_\xi^\vee}{\lra} X(T_{f_\xi}/F_\cyc)
\end{equation}
 \begin{itemize}
 \item  Then the kernel and the cokernel of $s_\xi^\vee$ are finitely generated $\Z_p$ modules for every   $\xi \in \mathfrak X (\RR)$ and 
 \item  There exists a non zero ideal $J$ in $\RR$ such that for any $\xi \in \mathfrak X (\RR) \setminus S_J$, the kernel and the cokernel of $s_\xi^\vee$ are finite, where the set $S_{J}$ is defined as $S_J : = \{ \xi \in \mathfrak X (\RR) \mid P_\xi \text{ does not contain } J \}$.
 \end{itemize}
In particular, assuming {\bf ($\text{Tor}$)}, the equality $$Ch_{O_{f_\xi}[[T]]}(\x(\T_\RR/F_\cyc)/{P_\xi \x(\T_\RR/F_\cyc)}) = Ch_{O_{f_\xi}[[T]]}(X(T_{f_\xi}/F_\cyc))$$  holds for all $\xi \in \mathfrak X (\RR) \setminus S_J$.

\end{theorem}
\proof For a finitely generated $\RR$ module $M$, we define $M^\ddagger : = \text{Hom}_{\RR}(M, \RR)$. For convenience of notation, let $\T=\T_{\RR}$ and $\A=\A_{\RR}$. As before, $v_c$ will denote a prime in $F_\cyc$ lying above $S$ and  $I_c$ denotes  the inertia subgroup of $\bar{\Q}/F_\cyc$ with respect to the primes $\bar{v}/v_c$. We have the commutative diagram with the natural maps 
\begin{scriptsize}
\begin{equation}\label{splzc} 
 \xymatrix{  0 \ar[r] & S(  \mathcal A/F_\cyc)[P_\xi] \ar[r]  & H^1(F_S/F_\cyc , \mathcal A)[P_\xi]  \ar[r] &  \underset {v_c \mid S, v_c \nmid p} {\oplus}H^1( I_{v_c}, \mathcal A)[P_\xi] \underset { v_c \mid \p \mid p} {\oplus}H^1( I_{v_c}, \A_{\p}^{-}) [P_\xi] \\ 
                   0 \ar[r] & S(  A_{f_\xi}/F_\cyc) \ar[r] \ar[u]_{s_\xi}&  H^1(F_S/F_\cyc ,   A_{f_\xi}) \ar [r]  \ar[u]_{\eta_\xi}&   \underset {v_c \mid S, v_c \nmid p} {\oplus}H^1( I_{v_c}, A_{f_\xi}) \underset {v_c \mid \p \mid p} {\oplus}H^1( I_{v_c}, (A_{f_\xi})_{\p}^{-}) \ar[u]_{\delta_\xi = \oplus \delta_{v_c}^\xi}}
\end{equation}
\end{scriptsize}
Recall, $A_{f_\xi} \cong \mathcal A[P_\xi] $ and $(A_{f_\xi})_{\p}^{-} \cong \A_{\p}^{-}[P_\xi] $. 
Since we assume that the residual representation is absolutely irreducible, we have (by \cite[Remark 3.4.1]{g3})
\begin{equation}
 \xymatrix{ H^1(F_S/F_\cyc , \mathcal A[P_\xi] ) \ar[r]^{\eta_{\xi}} & H^1(F_S/F_\cyc , \mathcal A)[P_\xi] }
\end{equation} 
is an isomorphism. By Snake lemma, we get that $\text{ker}(s_\xi)$ is trivial for every $\xi$. Thus $\text{coker}(s_\xi^\vee) =0$.

Next we want to prove $\text{ker}(s_{\xi}^{\vee})$ is finitely generated $\Z_{p}$ module for every $\xi$. As there are only finitely many primes in $F_\cyc$ lying over a given prime in $F$, it is enough to show that $\text{ker}( \delta_{v_{c}}^{\xi})^{\vee}$ is finitely generated $\Z_{p}$ module for all $v_{c} \mid S$.
By our assumption, $\RR  \cong O[[X_{1},\cdots,X_{r}]]$, where $r= [F: \Q] +1+\delta_{F,p}$, where $\delta_{F,p}$ is the defect of Leopoldt's conjecture for $F$ at $p$  and $O$ is some finite extension of $\Z_p$.  Since $\RR$  is regular local and $P_{\xi}$ is a prime ideal of height $r$, we have $P_{\xi} = (x_{1},\cdots,x_{r})$, here $x_{1},\cdots,x_{r}$ is a regular sequence of prime elements of $\RR$. Define $P_{0} = (0)$ and $P_{i} = (x_{1},\cdots,x_{i})$, for $ 1 \leq i \leq r$.  Then $P_{i}$ is a prime ideal and $\A[P_{i}]$ is divisible as $\RR/P_{i}$ module. Notice that $\A[P_{i}] = (\A[P_{i-1}])[x_{i}]$, and multiplication by $x_{i}$ is surjective on $\A[P_{i-1}]$. We get a induced map of $\RR/P_{i}$ modules
$$0 \to \A^{I_{v_c}}[P_{i-1}]/ x_{i}\A^{I_{v_c}}[P_{i-1}] \to H^1( I_{v_c}, \mathcal A [P_{i}]) \to H^1( I_{v_\infty}, \mathcal A[P_{i-1}])[x_i] \to 0.$$
Thus we can obtain kernel of $ \delta_{v_{c}}^{\xi}$ via successive extensions. Let $Ker_{1} = \A^{I_{v_c}} /x_{1} \A^{I_{v_c}}$, and let $Ker_{i}$ denotes the kernel of the map $ H^1( I_{v_c}, \mathcal A [P_{i}]) \to H^1( I_{v_c}, \mathcal A)[P_{i}] $, then $Ker_{i}$ is an extension of the form, 
\begin{equation}\label{ker1}
 0 \to \A^{I_{v_c}}[P_{i-1}]/ x_{i}\A^{I_{v_c}}[P_{i-1}] \to Ker_{i} \to Ker_{(i-1)}[x_{i}] \to 0.
\end{equation}
Then in this notation $ \text{ker}(\delta_{v_{c}}^{\xi}) = \text{Ker}_{r}$. 

First, we show that $(\text{Ker}_{r})^{\vee}$ is finitely generated $\Z_{p}$ module for every $\xi$. We also denote $I_{v_c}$ by $G$ for the rest of  this proof.  Taking Pontryagin dual, we get from \eqref{ker1}, for each $i$,
\begin{equation}\label{ker2}
 0 \to \frac{\text{Ker}_{i-1}^\vee}{x_{i}\text{Ker}_{i-1}^\vee}  \to \text{Ker}^\vee_{i} \to \frac{\mathcal T^{\ddagger}_{G}}{P_{i-1}\mathcal T^{\ddagger}_G}[x_{i}] \to 0.
\end{equation}

Now for $i=1$, $\text{Ker}^\vee_{1} \cong  \mathcal T^{\ddagger}_G[x_{1}]$ is a finitely generated $\RR/{x_1}$ module. By induction, assume that for $i = 1, 2, \cdots, r-1$, $\text{Ker}^\vee_{i} $ is a finitely generated $\RR/{(x_1, \cdots, x_i)}$ module. Also, $\mathcal T^{\ddagger}_G$ being a finitely generated $\RR$ module, it is immediate that $\mathcal T^{\ddagger}_{G}/{P_{r-1}\mathcal T^{\ddagger}_G}$ is a finitely generated $\RR/P_{r-1}$ module.  Consequently,  $\mathcal T^{\ddagger}_{G}/{P_{r-1}\mathcal T^{\ddagger}_G}[x_r]$ is a finitely generated $\RR/P_{r}$ module. Also by induction hypothesis, $\text{Ker}_{r-1}^\vee /{x_{r}}$ is a finitely generated $\RR/{(P_{r-1}, x_r)} \cong R/P_r$ module.  Hence, we deduce from \eqref{ker2}, that $\text{Ker}_{r}^\vee$ is a finitely generated $\RR/P_r$ module. Recall, in this notation, $P_r = (x_1,\cdots, x_r) = P_\xi$ and hence $\RR/P_\xi \cong O,$ a finite extension of $\Z_p$. This finishes the proof of the first assertion of the theorem.

For the second assertion of the theorem, we have to show that there exists a non zero ideal $J$ in $\RR$ such that for any $\xi \in \mathfrak X (\RR)$ such that  $\text{ker}(\xi) = P_\xi \nsupseteq J$, the rank 
\begin{equation}\label{ker3}
\text{rk}_{\RR/(x_1, \cdots, x_r)} ~ \text{Ker}_r^\vee =0.
\end{equation}
 By \cite[Theorem 2.1]{g3}, corresponding to the $\RR$ module $\mathcal T^{\ddagger}_{G}$ there exists an ideal $J \neq 0$ such that the  following two equations hold 
 \begin{equation}\label{ker4}
 \text{rk}_{\RR/P} ~ \mathcal T^{\ddagger}_{G}[P] =0 \text{ for any prime ideal  P } \nsupseteq J 
 \end{equation}

 \begin{equation}\label{ker5}
 \text{rk}_{\RR/P} ~ \frac{\mathcal T^{\ddagger}_{G}}{P\mathcal T^{\ddagger}_{G}} = \text{rk}_{\RR} ~ \mathcal T^{\ddagger}_{G} \text{ for any prime ideal } P \nsupseteq J
 \end{equation} 
 We will go on to show that this $J$ will work for us. Let  $P_\xi = (x_1, \cdots, x_r)  \nsupseteq J$. Then as $(x_1) \nsupseteq J$, we have $$\text{rk}_{\RR/x_1} ~ \mathcal T^{\ddagger}_{G}[x_1] =0.$$ But $\text{Ker}^\vee_1 \cong \mathcal T^{\ddagger}_{G}[x_1]$. Thus we get that $$\text{rk}_{\RR/x_1} ~ \text{Ker}^\vee_1 =0.$$ Let us assume for $i= 1, 2, \cdots, r-1$, 
 \begin{equation}\label{ker6}
 \text{rk}_{\RR/(x_1, \cdots, x_i)} ~ \text{Ker}^\vee_i=0.
 \end{equation}
 Notice that 
\begin{align}\label{ker7}
   \text{rk}_{\RR/(x_1,\cdots, x_r)} ~ \frac{\mathcal T^{\ddagger}_{G}}{P_{r-1}T^{\ddagger}_{G}}[x_r] & =  \text{rk}_{\RR/(x_1,\cdots, x_r)} ~ \frac{\mathcal T^{\ddagger}_{G}}{(P_{r-1},x_r)T^{\ddagger}_{G}} - \text{rk}_{\RR/(x_1,\cdots, x_{r-1})} ~ \frac{\mathcal T^{\ddagger}_{G}}{P_{r-1}T^{\ddagger}_{G}} \nonumber \\
   &= \text{rk}_{\RR/(x_1,\cdots, x_r)} ~ \frac{\mathcal T^{\ddagger}_{G}}{P_rT^{\ddagger}_{G}} - \text{rk}_{\RR/(x_1,\cdots, x_{r-1})} ~ \frac{\mathcal T^{\ddagger}_{G}}{P_{r-1}T^{\ddagger}_{G}}.
  \end{align}
By our assumption that $P_r \nsupseteq J$, we deduce from \eqref{ker4} that  
$$\text{rk}_{\RR/P_r} ~\frac{\mathcal T^{\ddagger}_{G}}{P_r\mathcal T^{\ddagger}_{G}} = \text{rk}_{\RR/P_{r-1}} ~\frac{\mathcal T^{\ddagger}_{G}}{P_{r-1}\mathcal T^{\ddagger}_{G}} = \text{rk}_{\RR} ~\mathcal T^{\ddagger}_{G}.$$ 
Hence we deduce from \eqref{ker7} that   $\text{rk}_{\RR/P_r} ~ \frac{\mathcal T^{\ddagger}_{G}}{P_{r-1}T^{\ddagger}_{G}}[x_r] =0$. 

Thus using \eqref{ker2}, to finish the proof of the theorem, it suffices to show that  
\begin{equation}\label{ker8}
\text{rk}_{\RR/P_r} ~ \text{Ker}^\vee_{r-1}/{x_r }=0.
\end{equation}
 Now \begin{align}\label{ker9}
   \text{rk}_{\RR/(x_1,\cdots, x_r)} ~ \frac{\text{Ker}^\vee_{r-1}}{x_r \text{Ker}^\vee_{r-1}} & =  \text{rk}_{\RR/(x_1,\cdots, x_{r-1})} ~  \text{Ker}^\vee_{r-1} + \text{rk}_{\RR/(x_1,\cdots, x_{r})} ~ \text{Ker}^\vee_{r-1}[x_r]  \nonumber \\
   &= \text{rk}_{\RR/(x_1,\cdots, x_{r})} ~ \text{Ker}^\vee_{r-1}[x_r],   
  \end{align}
 where the last equality follows from \eqref{ker6}. Thus to prove \eqref{ker8}, we are further reduced to showing 
  \begin{equation}\label{ker10}
  \text{rk}_{\RR/(x_1,\cdots, x_{r})} ~ \text{Ker}^\vee_{r-1}[x_r]   =0
  \end{equation}
 Now we have the exact sequence, 
 \begin{equation}\label{ker11}
 0 \to \frac{\text{Ker}_{r-2}^\vee }{x_{r-1}\text{Ker}_{r-2}^\vee }[x_r]  \to \text{Ker}^\vee_{r-1}[x_r] \to \frac{\mathcal T^{\ddagger}_{G}}{P_{r-2}\mathcal T^{\ddagger}_G}[(x_{r-1}, x_r)]
\end{equation}
Proceeding similarly as in  \eqref{ker7}, we deduce  that 
\begin{align}\label{ker12}
 \text{rk}_{\RR/P_r} ~ \frac{\mathcal T^{\ddagger}_{G}}{P_{r-2}\mathcal T^{\ddagger}_G}[(x_{r-1}, x_r)] & =   \text{rk}_{\RR/P_r} ~ \frac{\mathcal T^{\ddagger}_{G}}{P_rT^{\ddagger}_{G}} - \text{rk}_{\RR/P_{r-2}} ~ \frac{\mathcal T^{\ddagger}_{G}}{P_{r-2}T^{\ddagger}_{G}} \nonumber \\
   &= 0.
  \end{align}
   
Thus   \begin{align}\label{ker13}
   \text{rk}_{\RR/P_r} ~ \text{Ker}^\vee_{r-1}[x_r]  & =  \text{rk}_{\RR/P_r} ~ \frac{\text{Ker}_{r-2}^\vee }{x_{r-1}\text{Ker}_{r-2}^\vee }[x_r]   \nonumber \\
   &= \text{rk}_{\RR/P_r} ~ \frac{\text{Ker}_{r-2}^\vee }{(x_{r-1}, x_r) \text{Ker}_{r-2}^\vee } - \text{rk}_{\RR/P_{r-1}} ~ \frac{\text{Ker}_{r-2}^\vee }{x_{r-1} \text{Ker}_{r-2}^\vee }.   
  \end{align}
 Now, note that 
 \begin{align}\label{ker14}
 \text{rk}_{\RR/P_r} ~ \frac{\text{Ker}_{r-2}^\vee }{(x_{r-1}, x_r) \text{Ker}_{r-2}^\vee } & =  \text{rk}_{\RR/P_{r-2}} ~ \text{Ker}_{r-2}^\vee  + \text{rk}_{\RR/P_r} ~ \text{Ker}_{r-2}^\vee[(x_{r-1}, x_r)] \nonumber \\ 
 & =\text{rk}_{\RR/P_r} ~ \text{Ker}_{r-2}^\vee[(x_{r-1}, x_r)].
 \end{align}
Here the last equality  follows from the induction hypothesis in \eqref{ker6}. Using this in \eqref{ker13}, we deduce that
\begin{align}\label{ker15}
   \text{rk}_{\RR/P_r} ~ \text{Ker}^\vee_{r-1}[x_r]  & = \text{rk}_{\RR/P_r} ~ \text{Ker}_{r-2}^\vee[(x_{r-1}, x_r)] - \text{rk}_{\RR/P_{r-1}} ~ \frac{\text{Ker}_{r-2}^\vee }{x_{r-1} \text{Ker}_{r-2}^\vee }  
 \end{align}
 Recall, from \eqref{ker2}, for $i=r-1$, we have the exact sequence
 \begin{equation}\label{ker16}
 0 \to \frac{\text{Ker}_{r-2}^\vee}{x_{r-1}\text{Ker}_{r-2}^\vee}  \to \text{Ker}^\vee_{r-1} \to \frac{\mathcal T^{\ddagger}_{G}}{P_{r-2}\mathcal T^{\ddagger}_G}[x_{r-1}] \to 0.
 \end{equation}
 Using induction hypothesis \eqref{ker6} in  \eqref{ker16}, we deduce that 
 $$\text{rk}_{\RR/P_{r-1}} ~ \frac{\text{Ker}_{r-2}^\vee }{x_{r-1} \text{Ker}_{r-2}^\vee }   = 0.$$
 Hence we obtain from \eqref{ker15} that 
 \begin{equation}\label{ker17}
  \text{rk}_{\RR/P_r} ~ \text{Ker}^\vee_{r-1}[x_r] = \text{rk}_{\RR/P_r} ~ \text{Ker}_{r-2}^\vee[(x_{r-1}, x_r)]
  \end{equation}
 Proceeding in a similar way, we deduce that 
\begin{align}\label{ker18}
\text{rk}_{\RR/P_r} ~ \text{Ker}^\vee_r & = \text{rk}_{\RR/P_r} ~ \text{Ker}^\vee_{r-1}[x_r] \nonumber \\ 
& = \text{rk}_{\RR/P_r} ~ \text{Ker}_{r-2}^\vee[(x_{r-1}, x_r)] \nonumber \\ 
& = \text{rk}_{\RR/P_r} ~ \text{Ker}^\vee_1[(x_2, \cdots, x_r)] \nonumber \\
& = \text{rk}_{\RR/P_r} ~ \mathcal T^{\ddagger}_{G}  [x_1][( x_2, \cdots,x_r)] \nonumber \\ 
& = \text{rk}_{\RR/P_r} ~ \mathcal T^{\ddagger}_{G}  [(x_1, x_2, \cdots,x_r)] \nonumber \\ 
& = \text{rk}_{\RR/P_r} ~ \mathcal T^{\ddagger}_{G}  [(x_1, x_2, \cdots,x_r)] \nonumber \\ 
& = \text{rk}_{\RR/P_r} ~ \mathcal T^{\ddagger}_{G} [P_r] =0.
\end{align}
 Here the last equality follows from our hypothesis that $P_r = P_\xi \nsupseteq J$. This finishes the proof of the second assertion of the theorem.
\qed

\begin{rem}\label{splzd}
Let the assumptions be as in Theorem \ref{spl}. Then proceeding as in  the  proof of Theorem \ref{spl}, we can  deduce the corresponding theorem for $\T_\RR^*$ i.e. there exists a non zero ideal $J^*$ in $\RR$ such that for any $\xi \in \mathfrak X (\RR) \setminus S_{J^*}$, the kernel and the cokernel of ${s^*}_\xi^\vee$ are finite, where  $S_{J^*} : = \{  \xi \in \mathfrak X (\RR) \mid P_\xi \text{ does not contain } J^* \}$.
 Consequently under the assumption {\bf ($\text{Tor})$}, $$Ch_{O_{f_\xi}[[\Gamma]]}(\x(\T^*_\RR/F_\cyc)/{P_\xi   \x(\T^*_\RR/F_\cyc)}) =Ch_{O_{f_\xi}[[\Gamma]]}( X(T^*_{f_\xi}/F_\cyc))$$   for every $\xi \in \mathfrak X(\RR) \setminus S_{J^*}$. 
\end{rem}

\begin{rem}\label{splzdd}
 Under the assumption {\bf ($\text{Tor}$)}, by applying the involution $\iota$ we obtain,
\begin{align*}
Ch_{O_{f_\xi}[[\Gamma]]}(\x(\T^*_\RR/F_\cyc)^\iota/{P_\xi   \x(\T^*_\RR/F_\cyc)^\iota}) = Ch_{O_{f_\xi}[[\Gamma]]}((\x(\T^*_\RR/F_\cyc)/{P_\xi   \x(\T^*_\RR/F_\cyc)})^\iota) \\
=  Ch_{O_{f_\xi}[[\Gamma]]}(X(T^*_{f_\xi}/F_\cyc)^\iota).
\end{align*}
 for every $\xi \in \mathfrak X(\RR) \setminus S_{J^*}$.
\end{rem}

\begin{lemma}\label{fgsnmlt}
Assume {\bf (Irr)}, {\bf (Dist)} and that $\RR $ is a power series ring in many  variables. Then $\mathcal X(\T_\RR/F_\cyc)$ and $\mathcal X(\T_\RR^*/F_\cyc)$ are finitely generated  $\RR[[\Gamma]]$ modules.
\end{lemma}
\begin{proof} This result  follows  from Theorem \ref{spl}, remark \ref{splzd}, Lemma \ref{fgsn} and  topological Nakayama's Lemma \cite[Corollary 5.2.18]{nsw}.    \end{proof}

\begin{corollary}\label{bselt}
Let the assumption be as in  Theorem \ref{spl}. Then under  hypothesis {\bf ($\text{Tor}$)},  $\x(\T_\RR/F_\cyc)$ and $\x(\T^*_\RR/F_\cyc)$ are finitely generated torsion  $\RR[[\Gamma]]$ modules. 

\end{corollary}

\proof We will prove  for $\x(\T_\RR/F_\cyc)$ and a similar argument works for $\x(\T^*_\RR/F_\cyc)$. Choose any $\xi \in \mathfrak X(\RR) \setminus S_J$. It suffices  to show the localization at $\RR \setminus P_\xi$, $\x(\T_\RR/F_\cyc)_{(P_\xi)}$ = 0.  By  {\bf (Tor)} and  Theorem \ref{spl}, $\frac{\x(\T_\RR/F_\cyc)}{P_\xi  \x(\T_\RR/F_\cyc)}$ is a torsion  $\RR/{P_\xi}$ module. From this, using  localization argument and Nakayama's Lemma, we get $\x(\T_\RR/F_\cyc)_{(P_\xi)}$ = 0.  \qed

\begin{proposition}\label{lift}
Let $M$ and $N$ be two finitely generated torsion modules over the $r +1 $  variable power series ring $R_O = O[[W, T]]$ where W=$(X_{1},\cdots,X_{r})$,  $O$  is the ring of integer of a finite extension of $\Q_p$. Let $\{l_i\}_{i \in \N}$ be an infinite set of co-height 1 prime ideals in $O[[W]]$ such that 
\begin{enumerate}
\item $O[[W]]/ l_i$ is a finite extension of $O$, for any $i$. 
\item For each $l_i$, both $M/{l_i M} $ and $N/{l_i N} $ are torsion over $R_O/l_i$ and 
\item for every $i$, the image of $Ch_{R_O}(M)$ (resp. $Ch_{R_O}(N)$) in  $R_O/l_i$ equals \\
$Ch_{R_O/l_i}(M/{l_i M})$ (resp. $Ch_{R_O/l_i}(N/{l_i N})$), as ideals in   $R_O/l_i$.
\end{enumerate}
Then the equality of the ideals in $Ch_{R_O/l_i}(M/{l_i M}) = Ch_{R_O/l_i}(N/{l_i N})$ in  $R_O/l_i$ for every $l_i$ implies the equality of ideals $Ch_{R_O}(M) = Ch_{R_O}(N)$ in $R_O$.

\end{proposition}
\proof

For $r=1$, the result is essentially continued in \cite[\S 3]{o3}. Suppose $l_{i} = (l_{i,1},l_{i,2},\cdots,l_{i,r})$, denote by $l_{i}^{(j)} = (l_{i,1},l_{i,2},\cdots,l_{i,j})$.  Note that, $M \otimes R_O/l_{i}^{(j)} = (M \otimes R_O/l_{i}^{(j-1)}) \otimes R/l_{i,j}$. We claim that, if $Ch_{R_{O}/(l_{i,1},l_{i,2},\cdots,l_{i,j})} M \otimes R_{O}/(l_{i,1},l_{i,2},\cdots,l_{i,j}) = Ch_{R_{O}/(l_{i,1},l_{i,2},\cdots,l_{i,j})} N \otimes R_{O}/(l_{i,1},l_{i,2},\cdots,l_{i,j})$ is true for infinitely many $l_{i}^{(j)}$, for which first $(j-1)$ generators are same (that is $l_{i}^{(j-1)}$'s are same for all $i$), then we have,
$$Ch_{R_{O}/(l_{i}^{(j-1)})} M \otimes R_{O}/(l_{i}^{(j-1)}) = Ch_{R_{O}/(l_{i}^{(j-1)})} N \otimes R_{O}/(l_{i}^{(j-1)}).$$
Hence we prove the result by applying this to $j=r,r-1,\cdots,1$.

We use multivariable notation $h(\mathbb{T})$ to denote polynomial $h(X_{1},\cdots,X_{r},T)$. As $M$ and $N$ are finitely generated torsion module over $(r+1)$ variable Iwasawa algebra, using the structure theorem of Iwasawa modules, we fix $R_{O}$ module pseudo-isomorphisms $\phi$ and $\psi$ respectively,
$$M \stackrel{\phi}{\lra}  \underset{i}{\oplus} R_O/\pi_O^{\mu_i} \underset{j}{\oplus} R_O/{h_j(\mathbb T)^{\lambda_j}} \quad  \text{and} $$
$$N \stackrel{\psi}{\lra}  \underset{i'}{\oplus} R_O/\pi_O^{\mu_{i'}'} \underset{j'}{\oplus} R_O/{g_{j'}(\mathbb T)^{\lambda_{j'}'}}. $$ 
Here $\pi_O$ is a uniformizing parameter for $O$. Set $h(\mathbb T) = \prod h_j(\mathbb T)^{\lambda_j}$, $g(\mathbb T) = \prod g_{j'}(\mathbb T)^{\lambda_{j'}}$ and $\mu = \sum  \mu_i $,  $\mu' = \sum  \mu_{i'} $.

We will show that  $Ch_{R_O/l_{i}^{(j-1)}}(N \otimes R_{O}/l_{i}^{(j-1)}) \subset  Ch_{R_O/l_{i}^{(j-1)}}(M \otimes R_{O}/l_{i}^{(j-1)})$. Interchanging $M$ and $N$, we will get the equality.   Clearly, it suffices to show that  the image of $\pi_O^{\mu'}$ is zero in $(R_{O}/l_{i}^{(j-1)})/{\pi_O^{\mu}}$ and the image of $g(\mathbb T)$ is zero in $(R_O/l_{i}^{(j-1)})/{h(\mathbb T)}$.

If $h(\mathbb T)$ is a unit in $R_O/l_{i}^{(j-1)}$ then obviously the image of  $g(\mathbb T)$ in $(R_O/l_{i}^{(j-1)})/{h(\mathbb T)}$ is zero. So we assume that $h(\mathbb T)$ is not a unit in $R_O/l_{i}^{(j-1)}$.  Then by  \cite[Lemma 3.8]{o3} there is a finite extension $O''$ of  $O$ such that by a change of coordinate by a linear transform, we may assume that $h(\mathbb T) = u(\mathbb T)f(T)$ where $u(\mathbb T)$ is a unit in $R_{O''}$ and $f \in O''[[W]][T]$. Now, if necessary, we move to an extension of $O''$ containing both $O' $ and $O''$ and denote again it by $O'$ (abusing the notation, just to ease the burden of  notation) such that $R_{O'}/l_i \cong O'[[T]]$. Then, the image of $g (\mathbb T)$ vanishes in  $R_{O'}/{(h(\mathbb T), l_{i}^{(j)})}$ for every $i$. 

For every $k \geq 1$, we have an injection

$$ (R_{O'}/l_{i}^{(j-1)})/(l_{1,j}l_{2,j} \cdots l_{k,j}) \hookrightarrow \prod_{1 \leq i \leq k} (R_{O'}/l_{i}^{(j-1)})/ (l_{i,j}) \cong \prod_{1 \leq i \leq k} R_{O'}/l_{i}^{(j)}$$

Since $R_{O'}/h(\mathbb{T})$ is finite flat over $O'[[W]]$, we get for each $k \geq 1$ an injection,

$$  (R_{O'}/l_{i}^{(j-1)})/(h(\mathbb{T}),l_{1,j}l_{2,j} \cdots l_{k,j}) \hookrightarrow \prod_{1 \leq i \leq k} R_{O'}/(h(\mathbb{T}),l_{i}^{(j)}).$$

We observe that image of $g(\mathbb{T})$ vanishes in $(R_{O'}/l_{i}^{(j-1)})/(h(\mathbb{T}),l_{1,j}l_{2,j} \cdots l_{k,j})$. Thus the image of $g(\mathbb T)$ is zero in $R_{O'}/(h(\mathbb T), l_{i}^{(j-1)})$, since
$$ \underset{k}{\varprojlim}~ (R_{O'}/l_{i}^{(j-1)})/(h(\mathbb{T}),l_{1,j}l_{2,j} \cdots l_{k,j}) \cong R_{O'}/(h(\mathbb T), l_{i}^{(j-1)}) .$$

 For the $\mu$ invariants, for any $l_{i}$, $\pi_{O'}^\mu $ (resp. $\pi_{O'}^{\mu'}$)  is equal to the highest power of $\pi_O$ dividing the characteristic power series of $M_{O'}/{l_i M_{O'}}$ (resp. $N_{O'}/{l_i N_{O'}}$), where $M_{O'} := M \otimes_O O'$ is the extension of scalers from $O$ to $O'$. Hence $\pi^\mu_{O'} = \pi^{\mu'}_{O'}$.  Thus it follows that $\pi_O^{\mu'}$ is zero in $(R_{O'}/l_{i}^{(j-1)})/{\pi_O^{\mu}}$.

\qed

\begin{proposition}\label{bigp}
 For any finitely generated $ \RR[[\Gamma]]$ module  $U$, let $U^0$ denotes the maximal pseudonull $ \RR[[\Gamma]]$  submodule of  $U$. Assume the hypothesis { \bf ($\text{Tor}$)}. Then  for every $\xi \in \mathfrak X (\RR) $, $\frac{ \mathcal X(\T_\RR/F_\cyc) ^0}{P_\xi \mathcal X(\T_\RR/F_\cyc)^0} $  (resp. $\frac{\mathcal X(\T^*_\RR/F_\cyc)^0}{P_\xi \mathcal X(\T^*_\RR/F_\cyc)^0} $) are pseudonull $O_{f_\xi}[[\Gamma]]$ modules. 
\end{proposition}
\proof  We broadly follow the same strategy as in \cite[Lemma 7.2]{o} but the arguments are different. We prove the result only for $\mathcal X(\T_\RR/F_\cyc) $ as an entirely similar argument holds for $\mathcal X(\T^*_\RR/F_\cyc).$  Recall, for $C = \mathcal A $ or $C= A_{f_\xi}$, with notation as before, the  Selmer group $S(C/F_\cyc)$ fits into the exact sequence 
\begin{scriptsize}
\begin{equation}\label{scyc}
0 \rightarrow S(C/F_\cyc) \lra H^1(F_S/F_\cyc , C)  \lra  \underset {v_c \mid S, v_c \nmid p} {\oplus}H^1(I_{v_c}, C)^{G_{ v_c}} \underset { v_{c} \mid \p \mid p} {\oplus}H^1( I_{v_c},  C_{\p}^{-})^{G_{ v_c}}  
\end{equation}
\end{scriptsize}
\n  Note,  for $v_c \nmid p$,
\begin{scriptsize}
\begin{equation}\label{deco}
0 \rightarrow H^1(G_{v_c}/I_{v_c}, C^{I_{v_c}}) \lra H^1(G_{v_c}, C) \lra H^1(I_{v_c}, C)^{G_{v_c}} \lra H^2(G_{v_c}/I_{v_c}, C^{I_{v_c}}) 
\end{equation} 
\end{scriptsize}
with $G_{v_c}/I_{v_c} \cong \underset{l \neq p}{\oplus}\Z_l$ and $C$ is a $p$-torsion group. Thus $H^i(G_{v_c}/I_{v_c}, C^{I_{v_c}}) =0$ for $i =1,2$. Thus we see that, $ H^1(G_{v_c}, C) \simeq H^1(I_{v_c}, C)^{G_{v_c}}$.

Combining these,  we can have the following alternative definition of $S(C/F_\cyc)$.
\begin{scriptsize}
\begin{equation}\label{bselca}
 0 \rightarrow S(C/F_\cyc) \lra H^1(F_S/F_\cyc , C)  \lra  \underset {v_c \mid S, v_c \nmid p} {\oplus} H^1(G_{v_c}, C) \underset { v_c \mid \p \mid p} {\oplus}H^1( I_{v_c},  C_{\p}^{-})^{G_{ v_c}}  
\end{equation}
\end{scriptsize}

Similarly, for strict Selmer group we obtain that,
\begin{scriptsize}
\begin{equation}\label{strselca}
 0 \rightarrow S^{str}(C/F_\cyc) \lra H^1(F_S/F_\cyc , C)  \lra  \underset {v_c \mid S, v_c \nmid p} {\oplus} H^1(G_{v_c}, C) \underset { v_c \mid \p \mid p} {\oplus}H^1( G_{v_c},  C_{\p}^{-})
\end{equation}
\end{scriptsize}

We have from \cite[\S 2.2.2]{fo} an exact sequence, $$ 0 \to S^{str}(C/F_\cyc) \to S(C/F_\cyc) \to \underset{\mathfrak{p}|p} {\oplus} H^{1}(G_{F_{\mathfrak{p}}}/I_{\mathfrak{p}}, (C_{\p}^{-})^{I_{\mathfrak{p}}}).$$

Moreover from the proof of \cite[Corollary 3.4]{fo}, when $C=\A$, we see that $S^{str}(\A/F_\cyc) = S(\A/F_{\cyc})$. 

Under the assumption {\bf (Tor)}, we have  that $X(T_{f_\xi}/F_\cyc)$ is torsion for any $\xi \in \mathfrak X (\RR)$. Also by corollary \ref{bselt}, $\mathcal X(\T_\RR/F_\cyc)$ is torsion over $\RR[[\Gamma]]$. It follows that the maps defining  $S( C/F_\cyc)$ in \eqref{strselca}   is surjective for $C =\A$ or $C= A_{f_\xi}$ with $\xi \in \mathfrak X (\RR)$(\cite[Corollary 4.12]{o}).  Now let us  consider the commutative diagram 
\begin{scriptsize}
\begin{equation}\label{bseld} 
 \xymatrix{  0 \ar[r] & S^{str}(  \mathcal A/F_\cyc) \ar[r]  & H^1(F_S/F_\cyc , \mathcal A)  \ar[r] &   \underset { v_c \mid \p \mid p} {\oplus}H^1(F_{\cyc,v_c}, \A_{\p}^{-}) \underset {v_c \mid S, v_c \nmid p} {\oplus}H^1( F_{\cyc, v_c}, \mathcal A) \ar[r] &  0\\ 
                   0 \ar[r] & S^{str}(  \mathcal A/F_\cyc) \ar[r] \ar[u]_{\times P_\xi}&  H^1(F_S/F_\cyc ,   \mathcal A) \ar [r]  \ar[u]_{\times P_\xi}& \underset {v_c \mid \p \mid p} {\oplus}  H^1(F_{\cyc,v_c}, \A_{\p}^{-}) \underset {v_c \mid S, v_c \nmid p} {\oplus}H^1( F_{\cyc, v_c}, \mathcal A)   \ar[r]  \ar[u]_{\times P_\xi} & 0.}
\end{equation}
\end{scriptsize}
Recall $A_{f_\xi} \cong \mathcal A[P_\xi]$ and $(A_{f_\xi})_{\p}^{-} \cong \A_{\p}^{-}[P_\xi]$ as Galois modules.  For any $\xi \in \mathfrak X (\RR)$,  as  $X(T_{f^*_\xi}/F_\cyc)$  is torsion by  {\bf (Tor)}, we get  that $H^2(F_S/F_\cyc,  A_{f_\xi}) =  0 $ (see \cite[Proposition 2.3]{hm}). Then the cokernel of the middle vertical map in \eqref{bseld}, being is a subgroup of $H^2(F_S/F_\cyc,  A_{f_\xi})$, vanishes.   Thus by applying a Snake lemma to the diagram \eqref{bseld}, we get that
$$ \frac{S^{str}(  \mathcal A/F_\cyc)}{{P_\xi S^{str}(  \mathcal A/F_\cyc)}} \cong \text{coker}( H^1(F_S/F_\cyc, \mathcal A)[P_\xi] \stackrel{l_\xi}{\lra} W[P_\xi]), $$
where $$W :=  \underset { v_c \mid \p \mid p} {\oplus} H^1(F_{\cyc,v_c}, \A_{\p}^{-}) \underset {v_c \mid S, v_c \nmid p} {\oplus}H^1( F_{\cyc, v_c}, \mathcal A).$$
Similarly define $$W_\xi : =\underset { v_c \mid \p \mid p}{\oplus} H^1(F_{\cyc,v_c}, (A_{f_\xi})_{\p}^{-}) \underset {v_c \mid S, v_c \nmid p} {\oplus}H^1( F_{\cyc, v_c}, A_{f_\xi}) .$$
Then we have the commutative diagram 
\begin{scriptsize}
\begin{equation}\label{bseldf} 
 \xymatrix{  0 \ar[r] & S^{str}(  \mathcal A/F_\cyc)[P_\xi] \ar[r]   & {H^1(F_S/F_\cyc ,   \mathcal A)[P_\xi]} \ar[r] & W[P_\xi]  \ar[r] & \frac{ S^{str}(  \mathcal A/F_\cyc)}{{P_\xi S^{str}(  \mathcal A/F_\cyc)} } \ar[r] & 0\\ 
                   0 \ar[r] &  S^{str}(   A_{f_\xi}/F_\cyc) \ar[r] \ar[u] & H^1(F_S/F_\cyc ,  A_{f_\xi})   \ar [r]  \ar[u]  &   W_\xi  \ar[r]  \ar[u]     & 0}
\end{equation}
\end{scriptsize}
We see that the natural map $W_\xi \lra W[P_\xi]$ above is surjective. From the diagram \eqref{bseldf}, we see that the natural map  $H^1(F_S/F_\cyc, \mathcal A)[P_\xi] \stackrel{l_\xi}{\lra} W[P_\xi])$ is surjective. Thus we obtain $ S^{str}(  \mathcal A/F_\cyc)/{P_\xi S^{str}(  \mathcal A/F_\cyc)} = 0$. Since $S^{str}(\A/F_\cyc)= S(\A/F_\cyc)$, we see that,  $ S(  \mathcal A/F_\cyc)/{P_\xi S(  \mathcal A/F_\cyc)} = 0$. In other words, $\mathcal X(  \mathcal T_\RR/F_\cyc)[P_\xi] =0$. Thus  $\mathcal X(\T_\RR/F_\cyc)^0[P_\xi] = 0$. In particular,  $\mathcal X(\T_\RR/F_\cyc)^0[P_\xi] $ is a pseudonull $O_{f_\xi}[[\Gamma]]$ module. But for any finitely generated pseudonull $\RR[[\Gamma]]$ module $M$, $M/{P_\xi}$ is a pseudonull $O_{f_\xi}[[\Gamma]]$ module if and only if $M[P_\xi]$ is so (\cite[Lemma 3.1]{o3}). Thus, the last fact in turn implies that $\frac{\mathcal X(\T_\RR/F_\cyc)^0}{{P_\xi \mathcal X(\T_\RR/F_\cyc})^0}$ is also a pseudonull $O_{f_\xi}[[\Gamma]]$ module.   \qed

 \begin{lemma}\label{IJ}
 Let $J$ be any non-zero ideal in $\RR$. Let $I \in \text{Spec} (\RR) \setminus \text{Spec} (\RR/J)$ (that is $I$ is a prime ideal in $\RR$, which does not contain $J$). There exists at most finitely many $z_{1},\cdots,z_{k}$ in $\RR$ with the following properties:
 \begin{enumerate}
 \item $(I,z_{i})$ are distinct prime ideals in $\RR$ for all $i=1,\cdots,k$.
 \item For $i \neq j$, $\bar{z_{i}} \nmid \bar{z_{j}}$ in $\RR/I$.
 \item $(I,z_{i}) \supset J$ for all $i=1,\cdots,k$. 
 \end{enumerate}
 \end{lemma}
 
 \proof
 First we claim that, $$\cap_{i=1}^{r} (I,z_{i})= (I,z_{1}z_{2}\cdots z_{r}).$$
 Obviously, $(I,z_{1}z_{2}\cdots z_{r}) \subseteq \cap_{i=1}^{r} (I,z_{i})$. Now, let $x \in \cap_{i=1}^{r} (I,z_{i})$. Then,
 $$x = i_{1}+a_{1}z_{1}=\cdots=i_{r}+a_{r}z_{r},$$
 with $i_{j} \in I$ and $a_{j} \in \RR$. Then we see that,
 $$\bar{x}=\bar{a_{1}} \bar{z_{1}}=\cdots = \bar{a_{r}} \bar{z_{r}} \in \RR/I.$$
 Since $\bar{z_{i}} \nmid \bar{z_{j}}$, we see that $\bar{x} = \bar{\alpha} \bar{z_{1}}\cdots \bar{z_{r}} \in \RR/I$. Thus, $x = i+ \alpha z_{1}\cdots z_{r}$, where $i \in I$ and $\alpha$ is some lift of $\bar{\alpha}$ in $\RR$, which completes the proof of the claim.
 
 Suppose that there are infinitely many $z_{i}$'s in $\RR$ which satisfies all three properties. Then we have a decreasing chain of ideals in $\RR$,
 $$(I,z_{1}) \supset (I,z_{1}z_{2}) = \cap_{i=1}^{2} (I,z_{i}) \supset \cdots \supset (I,z_{1}\cdots z_{r}) = \cap_{i=1}^{r} (I,z_{i}) \supset \cdots .$$
 By assumption, $(I,z_{i})$ contains the ideal $J$ for all $i$. Thus we obtain,
 $$J \subseteq \cap_{i=1}^{\infty} (I,z_{i}) = \underset{r}{\varprojlim}~ (I,z_{1}\cdots z_{r}) \cong I,$$
 which is a contradiction to our assumption $I \in \text{Spec} (\RR) \setminus \text{Spec} (\RR/J)$.
 \qed
 
\begin{rem}\label{egl}
Let $M = \mathcal X(\T_\RR/F_\cyc) $ and $N =\mathcal X(\T^*_\RR/F_\cyc)^\iota $. Assume {\bf (Irr)}, {\bf (Dist)}, {\bf (Tor)} and $\RR \cong O[[W]]$. Let $S_1$ be the subset of arithmetic points for which $f_\xi$'s are exceptional (as defined in Theorem \ref{ctcyc}). Also let $S_2$ be the subset of $\mathfrak X (\RR)$ for which at least one among $\text{ker}(s_\xi)$,  $\text{coker}(s_\xi)$, $\text{ker}(s^*_\xi)$ and $\text{coker}(s^*_\xi)$ associated to the the natural specialization map $s_\xi$ in Theorem \ref{spl} and $s_\xi^*$ in remark \ref{splzd} is infinite. Define $S = S_1 \cup S_2$. Put  $ \mathfrak X (\RR)': = \mathfrak X (\RR) \setminus S$, then $ \mathfrak X (\RR)'$ is infinite.  

 Take the set $\{l_i\}_{i \in \N} = \mathfrak X (\RR)'$. Then  from Theorem \ref{spl}, corollary \ref{bselt}, Proposition \ref{bigp} and Lemma \ref{IJ}, we deduce that for these choices of $M$, $N$ and $l_i$'s,  all the condition of Proposition \ref{lift} are  satisfied (Lemma \ref{IJ} ensures us that for each $j$ in the proof of Proposition \ref{lift}, we have infinitely many ideals $l_{i}^{(j)}$ for which first $(j-1)$ generators are the same).  
\end{rem}

\begin{theorem}\label{mnm}
Let the notation be as before.  Let $F$ be a totally real number field, with $\Gamma = \text{Gal}(F_\cyc/F) \cong \Z_p $. Assume 
\begin{enumerate}
\item {\bf (Irr)}:  The residual representation $\bar{\rho}_{\mathcal{R}}$ of $G_F$   is absolutely irreducible.
\item {\bf (Dist)}: The restriction  of the residual representation at the decomposition subgroup i.e. $ \bar{\rho}_{\mathcal{R}} \mid_{G_{\p}}$ is an extension of two distinct characters of $G_{\p}$ with values in $\mathbb F_\RR^{\ast}$ for each $\p|p$.
\item {\bf (Tor)}:  For any normalized cuspidal Hilbert eigenform $f$, $X(T_f/F_\cyc)$ is a finitely generated torsion $O_f[[\Gamma]]$ module.
\item  $\RR$ is a power series ring.
\end{enumerate}
Then the functional equation holds for $\x(\mathcal T_\RR/F_\cyc)$ i.e. as an ideal in $\RR[[\Gamma]]$,
$$Ch_{\RR[[\Gamma]]}({\x(\T_\RR/F_\cyc)}) =Ch_{\RR[[\Gamma]]}({ \x(\T_\RR^*/F_\cyc)^\iota} ).$$ 

\end{theorem}
\proof 
By corollary  \ref{bselt}, $\x(\mathcal T_\RR/F_\cyc)$ and $\x(\mathcal T^*_\RR/F_\cyc)$ are torsion $\RR[[\Gamma]]$ modules. Using remark \ref{egl}, choose the infinite subset  $\mathfrak X (\mathcal R)'$ of arithmetic points.  By corollary \ref{bselt}, for every $\xi \in \mathfrak X (\mathcal R)'$,  $X(T_{f_{\xi}}/F_\cyc)$ and $X(T^*_{f_{\xi}}/F_\cyc)^\iota$ are  torsion over $O_{f_{\xi}}[[\Gamma]]$. 

\n Then applying Proposition  \ref{lift} for $M = \mathcal X(\T_\mathcal R/F_\cyc), N =\mathcal X(\T^*_\mathcal R/F_\cyc)^\iota $ and $\{l_i\}_{i \in \N} =  \mathfrak  X(\RR)'$, to prove  the theorem it suffices to show that  for every $\xi \in \mathfrak  X(\RR)'$, $$Ch_{O_{f_\xi}[[\Gamma]]}(\frac{\x(\T_\RR/F_\cyc)}{P_\xi  \x(\T_\RR/F_\cyc)})=Ch_{O_{f_\xi}[[\Gamma]]}(\frac{ \x(\T_\RR^*/F_\cyc)^\iota}{P_\xi \x(\T_\RR^*/F_\cyc)^\iota})$$ considered  as ideals in $O_{f_\xi}[[\Gamma]]$. By Theorem \ref{spl}, remark \ref{splzd} and remark \ref{splzdd} this in turn equivalent to showing  $$Ch_{O_{f_\xi}[[\Gamma]]}(X(T_{f_\xi}/F_\cyc))=Ch_{O_{f_\xi}[[\Gamma]]}(X(T^*_{f_\xi}/F_\cyc)^\iota)$$  for each $\xi \in \mathfrak  X(\RR)'.$ Hence we are done by Theorem \ref{fefibre}. \qed


\section{ Results over $\Z_p^2$ extension }\label{sec6}
Let $ K_\infty /K$ be the unique $\Z_p^{\oplus2}$  extension of an imaginary quadratic field $K$.  In this section, we will assume throughout that $p$ splits in $K$ and $D_K $  the discriminate of the imaginary quadratic field $K$, is coprime to tame conductor $N_{\mathcal R}$ of the branch $\RR$ of the Hida family i.e. $(p, D_K) = (D_K, N_{\mathcal R}) =  (N_\mathcal R, p) = 1 $. Recall the notation, $\Gamma_K = \text{Gal}(K_\infty/K) \cong \Z_p^2 $, $\Gamma = G(K_\cyc/K) \cong \Z_p$ and $ H = \text{Gal}(K_\infty/K_\cyc) \cong \Z_p$ so that $G/H \cong \Gamma$.


Recall from remark \ref{seleliptic}, $S = S_K$ is  a finite set of primes of $K$ dividing $Np$. Let $v$  be a prime of $K$ in $S$. Denote by $v_c$ a prime of $K_\cyc$ lying above  $v$  and let $v_\infty$  be a  prime of $K_\infty$ lying above $v_c$.  Let $\bar{v}$ be a prime of $\bar{\Q}$ lying above a prime $v_\infty $. Let $G_{ v_c}$  and  $G_{ v_\infty}$ denotes the decomposition subgroup  of $\bar{\Q}/K_\cyc$ and $\bar{\Q}/K_\infty$ for the prime $\bar{v}/{v_c}$ and  $\bar{v}/v_\infty$ respectively. Let   $I_{v_c}$,   $I_{v_\infty}$ denote the inertia group of $\bar{\Q}/K_\cyc$ and $\bar{\Q}/K_\infty$ for the prime $\bar{v}/{v_c}$ and  $\bar{v}/v_\infty$ respectively.  Let $\Gamma_{K_v}$ (resp. $H_{v_c}$) denote the decomposition subgroup of $\Gamma_K$ (resp. $H$ ) with respect to primes $v_\infty/v$ (resp. $v_\infty/v_c$). We  write $ I_{v_\infty/{v_c}}$  for the inertia subgroup of $K_\infty/K_\cyc$ with respect to prime $v_\infty/v_c$.

\begin{proposition}\label{ctcm}
The kernel of   the map $$X(T_f/K_\infty)_{H} \stackrel{\alpha_H^\vee}\lra X(T_f/K_\cyc)$$ is a finitely generated $\Z_p$ module and the cokernel  of $\alpha_H^\vee$ is  finite.
\end{proposition}
\proof 

 Set 
\[
 J_{v_c} : =
  \begin{cases}
 H^1(I_{v_c},  A_f)^{G_{v_c}}  & \text{if } v_c \mid S, v_c \nmid p\\
  H^1(I_{v_c},  A_f^-)^{G_{ v_c}}  & \text{if } v_c \mid p,
  \end{cases}
\]
\[
 J^\infty_{v_c} : =
  \begin{cases}
 \underset{v_\infty \mid v_c} \prod H^1(I_{v_\infty},  A_f)^{G_{ v_\infty}}  & \text{if } v_c \mid S, v_c \nmid p\\
  \underset{v_\infty \mid v_c} \prod H^1(I_{v_\infty},  A_f^-)^{G_{ v_\infty}}  & \text{if } v_c \mid p,
  \end{cases}
\]
  Using  similar argument as in \eqref{deco}, we get that
\[
 J_{v_c} : =
  \begin{cases}
  H^1(K_{\cyc, v_c}, A_f)  & \text{if } v_c \mid S, v_c \nmid p\\
  H^1(I_{v_c},  A_f^-)^{G_{ v_c}}  & \text{if } v_c \mid p,
  \end{cases}
\]
Also,
\[
 J^\infty_{v_c} : =
  \begin{cases}
 \underset{v_\infty \mid v_c} \prod H^1(K_{\infty, v_\infty},  A_f) \cong (\text{Ind}^{H}_{H_{v_c}} H^1(K_{\infty, v_\infty},  A_f)^\vee)^\vee & \text{if } v_c \mid S, v_c \nmid p\\
  \underset{v_\infty \mid v_c} \prod H^1(I_{v_\infty},  A_f^-)^{G_{ v_\infty}}   \cong (\text{Ind}^{H}_{H_{v_c}} {H^1(I_{v_\infty},  A_f)^{G_{v_\infty}}}^\vee)^\vee & \text{if } v_c \mid p,
  \end{cases}
\]
Then we have the commutative diagram
\begin{equation}\label{ctlcm} 
\xymatrix{  0 \ar[r] & S(  A_f/K_\infty)^{H} \ar[r]  & H^1(K_S/K_\infty , A_f)^{H}  \ar[r] & (\underset {v_c \mid S} {\prod} J^\infty_{v_c} )^{H} \\ 
                   0 \ar[r] & S(  A_f/K_\cyc) \ar[r] \ar[u]_{\alpha_H}&  H^1(K_S/K_\cyc ,  A_f) \ar [r]  \ar[u]_{\phi_H}&   \underset {v_c \mid S} {\prod} J_{v_c} \ar[u]_{\delta_H = \prod \delta^{v_c} }}
\end{equation}
By inflation-restriction sequence of Galois cohomology, the kernel of $\phi_H$ is isomorphic to $H^1(H, A_f^{G_{K_\infty}}).$  Clearly $U_\infty : = {A_f^{G_{K_\infty}}}^\vee$ is a finitely generated $\Z_p$ module. Thus we deduce that $\text{ker}(\phi_H)^\vee$ is a finitely generated $\Z_p$ module.  Moreover, as $H \cong \Z_p$, we see that $ H^1(H, A_f^{G_{K_\infty}}) $ is finite if and only if $H^0(H, A_f^{G_{K_\infty}}) \cong A_f^{G_{K_\cyc}}$ is finite. The last fact follows from (cf. \cite[Theorem A 2.8]{cs}, \cite[Lemma 2.1]{su}). Hence,  by Snake lemma on diagram \eqref{ctlcm}, we deduce that  $\text{cocker}(\alpha_H^\vee) $ is finite.

Also as $H$ has $p$-cohomological dimension $=1$, $\phi_H$ is surjective. Thus $\text{ker}(\alpha_H)^\vee$ is a subquotient  of the Pontryagin dual of kernel of $\delta_H$. Given a prime $v_c$ we pick  any one prime $v_\infty$ in $K_\infty$. Then by Shapiro's Lemma for each $v_c$, 
\[
 H^1(H, J_{v_c}^\infty)  \cong
  \begin{cases}
  H^1(K_{\infty, v_\infty},  A_f)  & \text{if } v_c \mid S, v_c \nmid p\\
  H^1(I_{v_\infty},  A_f^-)^{G_{ v_\infty}}  & \text{if } v_c \mid p,
  \end{cases}
\]
Then 

\[
 \ker(\delta^{v_c}) \cong 
  \begin{cases}
 H^1(H_{v_c},  A_f^{G_{ v_\infty}})  & \text{if } v_c \mid S, v_c \nmid p\\
 \text{ a subgroup of } H^1(I_{v_\infty/{v_c}} , {A_f^-}^{I_{v_\infty}}) & \text{if } v_c \mid p,
  \end{cases}
\]
 where the inertia subgroup $ I_{v_\infty/{v_c}} \cong \Z_p$. As before, we conclude that the Pontryagin dual of $\text{ker}(\delta^{v_c})$ is a finitely generated $\Z_p$ module.  From these, summing over finitely many primes and using a Snake lemma on diagram \eqref{ctlcm}, we deduce that $\text{ker}(\alpha_H^\vee)$ is  finitely generated over $\Z_p$.  \qed

\begin{rem}\label{torka}
By Kato's result (see \cite{ka}), we know that for any  $p$-stabilized newform $f$, $X(T_f/K_\cyc)$ (and $X(T^*_f/K_\cyc)$)  are finitely generated   torsion $\Z_p[[\Gamma]]$ modules.
\end{rem}

\begin{corollary}\label{seltm} 
 For any $p$-stabilized newform $f$, $X(T_f/K_\infty)$ and $X(T^*_f/K_\infty)$ are finitely generated {\bf torsion}   $O_f[[\Gamma_K]]$ modules.
\end{corollary}
\proof  By remark \ref{torka}, we have $X(T_f/K_\cyc)$ is a finitely generated torsion $O_f[[\Gamma]]$ module. By Proposition \ref{ctcm}, we also have an exact sequence $$0 \lra F_1 \lra X(T_f/K_\infty)_H \lra X(T_f/K_\cyc) \lra F_2 \lra 0$$ with $F_1$ is a finitely generated $\Z_p$ modules and $F_2$ is finite. Thus we get that $X(T_f/K_\infty)_H $ is a finite generated torsion $O_f[[\Gamma]]$ module. But for an ideal $I \neq O_f[[\Gamma_K]]$ in $O_f[[\Gamma_K]]$ 
$$ \text{rank}_{\frac{O_f[[\Gamma_K]]}{I}} \quad \quad \frac{X(T_f/K_\infty)}{IX(T_f/K_\infty)}  \geq \text{rank}_{O_f[[\Gamma_K]]} ~X(T_f/K_\infty) .$$ Identifying $O_f[[\Gamma_K]]$ with $O_f[[T_1, T_2]]$ and $X(T_f/K_\infty)_H$ with $\frac{X(T_f/K_\infty)}{T_2 X(T_f/F_\infty)}$, we deduce that $\text{rank}_{O_f[[\Gamma_K]]} ~X(T_f/K_\infty)  =0$. The argument for $X(T^*_f/F_\infty)$ is  similar. \qed

\begin{defn}
We call  $f \in S_2(\Gamma_0(Np),\psi)$ exceptional if $f$ is a newform of conductor $Np$ with $(N,p) = (\text{conductor of } \psi, p) =1$.
\end{defn}

\begin{theorem}\label{fefibrem}
Let the notation be as before.  Let $f \in S_k(\Gamma_0(Np^r),\psi)$ be a $p$-stabilized newform which is not exceptional.   Also assume that $(N, D_K) = (p, D_K) =1$.  Then the functional equation holds for $X(T_f/K_\infty)$ i.e. we have a equality of  ideals in $O_f[[\Gamma_K]]$, $$Ch_{O_f[[\Gamma_K]]}({X(T_f/K_\infty)}) =Ch_{O_f[[\Gamma_K]]}({ X(T^*_f/K_\infty)^\iota}) .$$
\end{theorem}
\proof   Note as the $p$-ordinary, $p$-stabilized newform  $f$ is not exceptional and we have, $(N, D_K) = (p, D_K) =1$;  the Galois representation $( \rho_f, V_f)$ satisfies both $(\text{Hyp}(K_\infty, V_f))$ and  $(\text{Hyp}(K_\infty, V^\vee_f))$ assumptions of \cite[Theorem 4.2.1]{pr}. Also, by corollary \ref{seltm}, we get that both  $X(T_f/K_\infty)$ and $X(T^*_f/K_\infty)^\iota$ are torsion over $O_f[[\Gamma_K]]$. Thus the condition $(\text{Tors}(K_\infty, V_f))$  in \cite[Theorem 4.2.1]{pr} is also satisfied.  Hence the theorem  follows  from \cite[Theorem 4.2.1]{pr}. \qed 

\medskip

Let us recall  from remark \ref{elipticcase}, $\mathcal T_\mathcal R$ is a lattice associated to a fixed branch $\mathcal R$ of the ordinary Hida family of elliptic modular form of tame level $N = N_\mathcal R$.

\begin{proposition}\label{splm}
Assume {\bf (Irr)}, {\bf (Dist)} and $\mathcal R $ is a power series ring. Then the kernel and the cokernel of the natural specialization map 
\begin{equation}\label{splzm}
\x(\T_\mathcal R/K_\infty)/{P_\xi   \x(\T_\mathcal R/K_\infty)} \stackrel{\beta_\xi^\vee}{\lra} X(T_{f_\xi}/K_\infty)
\end{equation}
 are pseudonull $O_{f_\xi}[[\Gamma_K]] \cong O_{f_\xi}[[T_1, T_2]]$  module  for all but finitely many $\xi \in \mathfrak X(\mathcal R)$. In particular,  the equality $$Ch_{O_{f_\xi}[[\Gamma_K]]}({\x(\T_\mathcal R/K_\infty)/{P_\xi \x(\T_\mathcal R/K_\infty)}}) = Ch_{O_{f_\xi}[[\Gamma_K]]}({X(T_{f_\xi}/K_\infty)})$$ as ideals in $O_{f_\xi}[[\Gamma_K]]$ holds for all    but finitely many arithmetic points.

\end{proposition}

\proof The proof is different from the proof of  Theorem \ref{spl} in two points. Here, we have to overcome the difficulty that there are possibly infinitely many primes in $K_\infty$ lying above a given prime in $K$. On the other hand, here  we have the advantage that $ \mathcal R \cong O[[W]]$, so that $P_\xi$ is principal.

Let us keep the  notation as set up in the beginning of section \ref{sec6}. 
For  a finitely generated $O_{f_\xi}[[{\Gamma_K}_v]]$ module $M$,  recall $\text{Ind}_{{\Gamma_K}_v}^{\Gamma_K} M := O_{f_\xi} [[\Gamma_K]] \otimes_{O_{f_\xi}[[{\Gamma_K}_v]]} M$. Recall, by Shapiro's lemma $H_i(\Gamma_K, \text{Ind}_{{\Gamma_K}_v}^{\Gamma_K} M) \cong H_i({\Gamma_K}_v, M)$ for any $i \geq 0$. 

Now, we have the commutative diagram with the natural maps 
\begin{equation}\label{splzcm} 
 \xymatrix{  0 \ar[r] & S(  \mathcal A/K_\infty)[P_\xi] \ar[r]  & H^1(K_S/K_\infty , \mathcal A)[P_\xi]  \ar[r] &  \underset {v \in S} {\prod}J_v(\mathcal A)[P_\xi]  \\ 
                   0 \ar[r] & S(  A_{f_\xi}/K_\infty) \ar[r] \ar[u]_{\beta_\xi}&  H^1(K_S/K_\infty ,   A_{f_\xi}) \ar [r]  \ar[u]_{\eta_\xi}&   \underset {v \in S} {\prod}J_v(\mathcal A_{f_\xi})  \ar[u]_{\delta_\xi = \prod \delta_v} .}
\end{equation}
Here \[
 J_v(B) : =
  \begin{cases}
   (\text{Ind}_{{\Gamma_K}_v}^{\Gamma_K} H^1(I_{v_\infty}, B)^\vee_{G_{{v_\infty}}} )^\vee & \text{if } v \in S, v\nmid p\\
   (\text{Ind}_{{\Gamma_K}_v}^{\Gamma_K} H^1(I_{v_\infty}, B^-)^\vee_{G_{{v_\infty}}} )^\vee & \text{if } v|p,
  \end{cases}
\]
for $B = \mathcal A$ or $B = A_{f_\xi}.$ Recall, once again $A_{f_\xi} \cong \mathcal A[P_\xi] $ and $A^-_{f_\xi} \cong \mathcal A^-[P_\xi] $. Also by our assumption that $\mathcal R$ is a power series ring, we have every $P_\xi$ is principal ideal in $\mathcal R$. Then  $\eta_\xi$ is surjective with $\text{ker}(\eta_\xi) \cong \mathcal A^{G_{K_\infty}}/{P_\xi A^{G_{K_\infty}}}$.  To simplify notation, we put $\mathcal T : = \mathcal T_\mathcal R$ in the proof of this theorem. Then $\text{ker}(\eta_\xi) ^\vee \cong (\mathcal {T}^\dagger_{G_{K_\infty}})_\text{tor}[P_\xi].$ Note, only finitely many $P_\xi$ divide the $\RR$ characteristic ideal of finitely generated torsion $\RR$ module $(\mathcal {T}^\dagger_{G_{K_\infty}})_\text{tor} $.  Hence we deduce that the $\text{ker}(\beta_\xi)^\vee$ is finite (hence $O_{f_\xi}[[{\Gamma_K}]]$ pseudonull) for all but finitely many $\xi \in \mathfrak X(\mathcal R).$

\n As before, it suffices to show  that $\text{ker}(\delta_\xi)^\vee$ is $O_{f_\xi}[[{\Gamma_K}]]$ pseudonull  leaving out  finitely many exceptional $\xi$.    
Now it is easy to see that \[
 (\text{ker}(\delta_v))^\vee  =
  \begin{cases}
   \text{Ind}_{{\Gamma_K}_v}^{\Gamma_K} ((\frac{{\mathcal A}^{I_{v_\infty}}}{P_\xi  {\mathcal A}^{I_{v_\infty}}})^{G_{v_\infty}} )^\vee  \cong  \text{Ind}_{{\Gamma_K}_v}^{\Gamma_K} ({\mathcal T}^\dagger_{I_{v_\infty}}[P_\xi])_{G_{v_\infty}} & \text{if } v \in S, v \nmid p\\
   \text{Ind}_{{\Gamma_K}_v}^{\Gamma_K} ((\frac{{\mathcal A^-}^{I_{v_\infty}}}{P_\xi  {\mathcal A^-}^{I_{v_\infty}}})^{G_{v_\infty}} )^\vee  \cong  \text{Ind}_{{\Gamma_K}_v}^{\Gamma_K}  ({\mathcal T^-}^\dagger_{I_{v_\infty}}[P_\xi])_{G_{v_\infty}} & \text{if } v|p,
    \end{cases}
\]
But for a prime $v$ in $S$ not lying above $p$, we have $ \mathcal T^\dagger_{I_{v_\infty}}[P_\xi]\cong  ({\mathcal T^\dagger_{I_{v_\infty}}})_{\text{tor}}[P_\xi] $  which is, as before, finite for all but finitely many $\xi$. Notice that  $K_\infty$ contains $K_\cyc$ and hence the dimension of ${\Gamma_K}_v$ as a $p$-adic Lie group is at least one.  Hence Krull dimension of the commutative ring $O_{f_\xi}[[{\Gamma_K}_v]]$ is at least $2$. Thus $  ({\mathcal T}^\dagger_{I_{v_\infty}}[P_\xi])_{G_{K_{\infty, v_\infty}}}$ is finite and hence pseudonull as an $O_{f_\xi}[[{\Gamma_K}_v]]$ module. Also, for a finitely generated $O_{f_\xi}[[{\Gamma_K}_v]]$ module $M$ and for $i = 0,1$ we have \cite[Lemma 2.7(i)]{ve}
\begin{scriptsize}
$$\text{Ext}^i_{O_{f_\xi}[[{\Gamma_K}]]} (O_{f_\xi}[[{\Gamma_K}]] \otimes_{O_{f_\xi}[[{\Gamma_K}_v]]}  M, O_{f_\xi}[[{\Gamma_K}]]) \cong O_{f_\xi}[[{\Gamma_K}]] \otimes_{O_{f_\xi}[[{\Gamma_K}_v]]} \text{Ext}^i_{O_{f_\xi}[[{\Gamma_K}_v]]} (  M, O_{f_\xi}[[{\Gamma_K}_v]]).$$
\end{scriptsize}
But a finitely generated $\La$ ( for $\La = O_{f_\xi}[[{\Gamma_K}]] $ or $O_{f_\xi}[[{\Gamma_K}_v]]$) module $M$ is pseudonull if and only if $\text{Ext}^i_\La(M, \La) =0 $ for $i=0,1$ (\cite{ve}). Thus  we see that $ \text{Ind}_{{\Gamma_K}_v}^{\Gamma_K} ({\mathcal T}^\dagger_{I_{v_\infty}}[P_\xi])_{G_{v_\infty}}$ is pseudonull as a $O_{f_\xi}[[{\Gamma_K}]]$ module for all but finitely many $\xi$ and for any $v \in S, v\nmid p$. The same argument holds for a prime $v$ in $S$ dividing $p$ if we replace  $\mathcal T$ by $\mathcal T^-$. Combining these for finitely many $v$ in $S$, we deduce that  $(\text{ker}(\delta_\xi))^\vee$ is a pseudonull $O_\xi[[{\Gamma_K}]]$ module for  all but finitely many  $\xi \in \mathfrak X(\mathcal R)$. This completes the proof. \qed

\begin{rem}\label{splzddm}
We have analogous  of   remarks \ref{splzd} and \ref{splzdd} for $K_\infty$ for the map $\x(\T^*_\mathcal R/K_\infty)/{P_\xi   \x(\T^*_\mathcal R/K_\infty)} \stackrel{\beta_\xi^{*^\vee}}{\lra} X(T_{f_\xi}/K_\infty).$  Proceeding in Proposition \ref{splm} and using Corollary \ref{seltm} we can get, $$Ch_{O_{f_\xi}[[\Gamma_K]]}({\x(\T^*_\mathcal R/K_\infty)/{P_\xi   \x(\T^*_\mathcal R/K_\infty)}}) =Ch_{O_{f_\xi}[[\Gamma_K]]}({ X(T^*_{f_\xi}/K_\infty)})$$ as ideals in $O_{f_\xi}[[\Gamma_K]]$  for all but finitely many $\xi \in \mathfrak X(\mathcal R).$  Further, applying the involution $\iota$, $$Ch_{O_{f_\xi}[[\Gamma_K]]}(\x(\T^*_\mathcal R/K_\infty)^\iota/{P_\xi   \x(\T^*_\mathcal R/K_\infty)^\iota}) =  Ch_{O_{f_\xi}[[\Gamma_K]]}({X(T^*_{f_\xi}/K_\infty)^\iota})$$ holds. Moreover, proceeding as in Corollary \ref{bselt}, we  deduce that both $\x(\T_\mathcal R/K_\infty)$ and $\x(\T^*_\mathcal R/K_\infty)$ are finitely generated {\bf torsion} modules over $\mathcal R[[\Gamma_K]]$.
\end{rem} 

\begin{proposition}\label{bigpm}
 For any finitely generated $ \mathcal R[[\Gamma_K]]$ module  $U$, let $U^0$ denotes the maximal pseudonull $ \mathcal R[[\Gamma_K]]$  submodule of  $U$. Then  $\frac{\mathcal X(\T_\mathcal R/K_\infty)^0}{P_\xi \mathcal X(\T_\mathcal R/K_\infty)^0}$ and  $\frac{\mathcal X(\T^*_\mathcal R/K_\infty)^0}{P_\xi \mathcal X(\T^*_\mathcal R/K_\infty)^0} $ are pseudonull $O_{f_\xi}[[\Gamma_K]]$ modules for any $\xi \in \mathfrak X (\mathcal R)$. 
\end{proposition}
\proof We proceed as in the proof of Proposition \ref{bigp}. Here the proof  is different  from Proposition \ref{bigp} as the strict Selmer group and the usual Selmer group for $\mathcal A$ over $K_\infty$ may not coincide.  As before, we will prove for  $\mathcal X(\T_\mathcal R/K_\infty)$ only.

As in Proposition \ref{bigp}, it suffices to show $\mathcal X(\T_\mathcal R/F_\infty) [P_\xi] =0$ for every $\xi$. for $B = \mathcal A  ~ \text{or} ~ A_{f_\xi}$, recall from section \ref{sec2}, the strict Selmer group 
  $S'(B/K_\infty) \hookrightarrow S(B/K_\infty)$. Then  by  corollary \ref{seltm} and remark \ref{splzddm}, we deduce that all four groups $X(T_{f_\xi}/K_\infty)$,  $X'(T_{f_\xi}/K_\infty)$, $ \mathcal X(\mathcal T_\mathcal R/K_\infty)$ and $\mathcal X'(\mathcal T_\mathcal R/K_\infty)$ are finitely generated  torsion modules over their respective Iwasawa algebras $O_{f_\xi}[[\Gamma_K]]$ and $\mathcal R[[\Gamma_K]]$.
It follows that  the maps defining the Greenberg Selmer groups and strict Selmer groups over $K_\infty$ are surjective  (\cite[Theorem 4.10]{o}, \cite[Theorem 7.12]{hv}). In other words,  $\phi'^B_{K_\infty}$ is surjective for $B = \A $ or $B =A_{f_\xi}$ and we have the exact sequence
$$ 0 \lra S(B/K_\infty ) \lra H^1(K_S/K_\infty, B) \lra \underset{v \in S, v\nmid p}\prod J_v(B/K_\infty) \underset{ v\mid p}\prod J^p_v(B^-/K_\infty) ) \lra 0$$
Here $J_v(B/K_\infty) = (\text{Ind}_{G_v}^G H^1(K_{\infty, v_\infty}, B)^\vee)^\vee$ and $J^p_v(B^-/K_\infty) = (\text{Ind}_{G_v}^G {H^1(I_{ v_\infty}, B)^{G_{v_\infty}}}^\vee)^\vee.$
 Hence we get another exact sequence $$ 0 \lra S'(B/K_\infty) \lra S(B/K_\infty) \lra U_p(B) \lra 0$$ 
with $U_p(B): = (\text{Ind}^{\Gamma_K}_{{\Gamma_K}_v} H^1(G_{v_\infty}/I_{v_\infty}, (B^-)^{I_{v_\infty}})^\vee)^\vee $. Therefore, it suffices to show that $\mathcal X'(\T_\mathcal R/F_\infty)[P_\xi] =0$ and $U_p(\A)^\vee[P_\xi] =0$. Now, 
\begin{equation}\label{csts}
H_1(G_{v_\infty}/I_{v_\infty}, {\A^-}^\vee)[P_\xi] = H_1(G_{v_\infty}/I_{v_\infty}, {\T_\mathcal R^\dagger}^-)[P_\xi],
\end{equation} where ${\T_\mathcal R^\dagger}^-$ is a free $\mathcal R$ module of rank 1. As before, write $<\text{Fr}_{v_\infty}> : = G_{v_\infty}/I_{v_\infty}$. Then the module in \eqref{csts} is isomorphic to $({\T_\mathcal R^\dagger}^-[\text{Fr}_{v_\infty}])[P_\xi]$. But $({\T_\mathcal R^\dagger}^-[\text{Fr}_{v_\infty}])[P_\xi] \subset ({\T_\mathcal R^\dagger}^-)[P_\xi] =0$ as ${\T_\mathcal R^\dagger}^-$ is a free $\mathcal R$ module. Using this, it is plain from the definition of  $U_p(\A)$ that $U_p(\A)^\vee[P_\xi] =0$. 

\n We will show $\mathcal X'(\T_\mathcal R/K_\infty)[P_\xi] =0$ to complete the proof of the proposition.  Set
 $$W'_\xi : = \underset{v \in S, v\nmid p}{\prod} (\text{Ind}^{\Gamma_K}_{{\Gamma_K}_v} H^1(K_{\infty, v_\infty}, A_{f_\xi})^\vee)^\vee  \underset{ v\mid p}\prod (\text{Ind}^{\Gamma_K}_{{\Gamma_K}_v} H^1(K_{\infty, v_\infty}, A_{f_\xi}^-)^\vee)^\vee \quad \text{and}$$ 
 $$\mathcal W' : = \underset{v \in S, v\nmid p}{\prod} (\text{Ind}^{\Gamma_K}_{{\Gamma_K}_v} H^1(K_{\infty, v_\infty}, \A)^\vee)^\vee  \underset{ v\mid p}\prod (\text{Ind}^{\Gamma_K}_{{\Gamma_K}_v} H^1(K_{\infty, v_\infty}, \A^-)^\vee)^\vee \quad \text{and}$$ 
 These are the `local factors'  used in the definition of  $X'(T_{f_\xi}/K_\infty)$ and  $\mathcal X'(\T_\mathcal R/K_\infty)$ respectively.  Once again, as in proposition \ref{bigp}, this follows if we can show the map $W'_\xi \lra \mathcal W'[P_\xi]$ is surjective. Recall  as $A[P_\xi] \cong A_{f_\xi}$. But for every prime $v_\infty$ of $K_\infty$ lying above a prime $S$ in $K$ not above $p$,  the map $H^1(K_{\infty, v_\infty}, \A[P_\xi]) \lra H^1(K_{\infty, v_\infty}, \A)[P_\xi]$  is surjective by Kummer theory. Similarly, for every $v_\infty \mid p$, $H^1(K_{\infty, v_\infty}, A^-_{f_\xi}) \lra H^1(K_{\infty, v_\infty}, \A^-)[P_\xi]$ is surjective. From these, it follows that $W'_\xi \lra \mathcal W'[P_\xi]$ is indeed surjective. This completes the proof of the proposition. \qed

\bigskip

\n We now state our main  theorem in  the  $\Z_p^2$ extension case. 
\begin{theorem}\label{mnmm}
Let the notation be as before.  Let $K$ be an imaginary quadratic field  and  $K \subset K_\cyc \subset K_\infty$ be the unique $\Z_p^{\oplus 2 }$ extension of $K$. Assume 
\begin{enumerate}
\item $(N_\mathcal R, D_K) = (p, D_K) =1$ i.e. $p$ splits in $K$ and the tame conductor $N_\RR$ associated to the branch of the Hida family is coprime to the discriminant of $K$.
\item {\bf (Irr)}  : The residual representation $\bar{\rho}_\mathcal R$ is absolutely irreducible as $G_\Q$-module.
\item {\bf (Dist)} : The characters appearing on the diagonal of the local residual representation $\bar{\rho}_\mathcal R \mid_{G_p}$ are distinct $(\text{mod } \mathfrak m)$.
\item  $\mathcal R$ is a power series ring.
\end{enumerate}
Then the functional equation holds for $\x(\mathcal T_\RR/K_\infty)$ i.e. as an ideal in $\mathcal R[[\Gamma_K]]$,
$$Ch_{\RR[[\Gamma_K]]}({\x(\T_\mathcal R/K_\infty)} )=Ch_{\RR[[\Gamma_K]]}({ \x(\T_\mathcal R^*/K_\infty)^\iota})$$ 

\end{theorem}
\proof 
By remark \ref{splzddm},  we see that $ \x(\T_\mathcal R/K_\infty)$ and $\x(\T_\mathcal R^*/K_\infty)$ are finitely generated torsion $\RR[[\Gamma_K]]$ module. Let $S_1$ be the finite subset of $\mathfrak X(\RR)$ consisting of those $\xi$ for which  $f_\xi$ is exceptional.  Let $S_2$ be the finite subset of $\mathfrak X(\RR)$ for which the map $\beta_\xi $ in Proposition \ref{splm} is not a pseudo-isomorphism. Similarly $S_3$ be the finite subset of $\mathfrak X(\RR)$ for which the map $\beta^*_\xi $ in the remark  \ref{splzddm} is not a pseudo-isomorphism. Let $\mathfrak S = S_1 \cup S_2 \cup S_3 $ be the finite  subset  of $\mathfrak X(\RR)$. Then by Proposition \ref{lift},  to prove the theorem, it suffices to show that for all $\xi \in  \mathfrak X(\RR) \setminus \mathfrak S$, 
\begin{equation}\label{last1}
Ch_{O_{f_\xi}[[\Gamma_K]]}({\x(\T_\mathcal R/K_\infty)}/{P_\xi {\x(\T_\mathcal R/K_\infty)}} )=Ch_{O_{f_\xi}[[\Gamma_K]]}({ \x(\T_\mathcal R^*/K_\infty)^\iota}/{P_\xi {\x(\T_\mathcal R/K_\infty)^\iota}}).
\end{equation}
 Using  Proposition \ref{splm} and  remark  \ref{splzddm}, \eqref{last1} is in  turn equivalent to showing  for all $\xi \in  \mathfrak X(\RR) \setminus \mathfrak S$, 
\begin{equation}\label{last}
Ch_{O_{f_\xi}[[\Gamma_K]]}(X(\T_{f_\xi}/K_\infty))=Ch_{O_{f_\xi}[[\Gamma_K]]}(X(\T_{f_\xi}^*/K_\infty)^\iota),
\end{equation}
which follows from Theorem \ref{fefibrem}. This completes the proof of the proposition. \qed

\begin{rem}

Our proof of    Theorem \ref{mnm} (respectively Theorem \ref{mnmm})  covers the case when $\x(\T_\mathcal R/F_\cyc)$ (respectively $\x(\T_\mathcal R/K_\infty)$) have non zero $\mathcal R[[\Gamma]]$  (respectively $\mathcal R[[\Gamma_K]]$) pseudonull submodules.  We also allow the Selmer groups $X(T_{f_{\xi}}/F_\cyc)$, $X(T_{f_{\xi}}/K_\cyc)$, $\xi \in \mathfrak X(\mathcal R)$ to have positive $\mu$-invariant. Our proof of Theorem \ref{mnmm} works for both CM or non-CM Hida family.
\end{rem}
\begin{rem}
 Following Mazur's conjecture, the crucial assumption {\bf (Tor)} for Hilbert modular forms in Theorem \ref{mnm}, is expected to be always true but has not been proven except in a  few special cases.  All the other conditions in Theorem \ref{mnm} (respectively all the conditions in Theorem \ref{mnmm}) are   satisfied in many cases.
\end{rem}
\begin{rem}
Keeping Leopoldt conjecture for a totally real field $F$ in mind, we do not pursue the $\Z_p^2$ extension case in case of Hilbert modular forms. 
\end{rem}


\begin{thebibliography}{9999999}

\bibitem[Bl]{bl} D. Blasius, Hilbert modular forms and the Ramanujan conjecture, Noncommutative geometry and number theory, {\it Aspects Math.}, E37 (2006) 35-56.

\bibitem[Ca]{ca} H. Carayol, Sur les repr\'esentations {$l$}-adiques associ\'ees aux formes modulaires de {H}ilbert, \emph{Ann. Sci. \'Ecole Norm. Sup.} (4), {\bf 19}, no. 3, (1986) 409-468.

\bibitem[C-S]{cs} J. Coates and R. Sujatha, Galois cohomology of elliptic curves, {\it Tata Inst. Fund. Res. Stud. Math.} {\bf 88} (2000).




\bibitem[Di]{di} M. Dimitrov, Automorphic symbols, {$p$}-adic {$L$}-functions and ordinary cohomology of Hilbert modular varieties. {\it Amer. J. Math.} {\bf 135} no. 4 (2013) 1117-1155.


\bibitem[F-O]{fo} O. Fouquet and T. Ochiai, Control Theorems for Selmer groups of nearly ordinary deformations,  {\it  J. Reine Angew. Math.}  {\bf 666} (2012), 163-187.

\bibitem[Gr1]{gr} R. Greenberg, Iwasawa theory for $p$-adic  representations,  Algebraic number theory,  {\it Adv. Stud. Pure Math.} {\bf 17} (1989) 97-137.
\bibitem[Gr2]{g3} R. Greenberg, On the  Structure of certain Galois cohomology group, \emph{Documenta Mathematica}, Extra Volume Coates (2006), 335-391

\bibitem[Hi1]{h1}  H. Hida, Galois representations into $\mathrm{GL}_2(\Z_p[[X]])$ attached to ordinary cusp forms, {\it Invent. Math.}  {\bf 85} (1986), no. 3, 545-613.
\bibitem[Hi2]{h2}  H. Hida, On p-adic Hecke algebras for GL2 over totally real fields, {\it Ann. of Math. (2)} {128} no. 2, (1988) 295-384. 
\bibitem[H-M]{hm} Y. Hachimori and K. Matsuno, An analogue of Kida's formula for the Selmer groups of elliptic curves, {\it J. Algebraic Geom.} {\bf 8} (1999), no. 3, 581-601.

\bibitem[H-V]{hv}  Y. Hachimori and O. Venjakob, Completely faithful Selmer groups over Kummer extensions., {\it Documenta Math.}, Kato Vol.  (2003) 443-478. 


\bibitem[J-P]{jp} S. Jha and A. Pal, {Algebraic functional equation for Hida family}, \emph{International Journal of Number Theory}, {\bf 10} (2014) No. 07, 1649-1674.



\bibitem[Ka]{ka}  K. Kato, $p$-adic Hodge theory and values of zeta functions of modular forms, Cohomologies $p$-adiques et applications arithm\'etiques. III.  {\it Ast\'erisque}   {\bf 295}  (2004), ix, 117-290.

\bibitem[M-W] {mw} B. Mazur and A. Wiles, On p-adic analytic families of Galois representations, {\it Compositio Math.} {\bf 59} (1986), 231-264.
\bibitem[Ma]{m} B. Mazur, Rational points of abelian varieties with values in towers of number fields, {\it Invent. Math.} {\bf 18}, (1972) 183-266.
\bibitem[M-SD]{msd} B. Mazur and P. Swinnerton-Dyer, Arithmetic of Weil Curves, {\it Invent. Math} {\bf 25} (1974) 1-61.
\bibitem[N-S-W]{nsw} J. Neukirch, A. Schmidt and  K. Wingberg, Cohomology of Number Fields,  Grundlehren der Mathematischen Wissenschaften, Springer-Verlag (2000).
\bibitem[Oc1]{o} T. Ochiai, On the two-variable Iwasawa main Conjecture, {\it Compositio. Math.}  {\bf 142} (2006), 1157-1200.

\bibitem[Oc2]{o2} T. Ochiai, Control Theorem for Greenberg's Selmer Groups of Galois Deformations, {\it Journal of Number Theory} {\bf 88} (2001), 59-85. 
\bibitem[Oc3]{o3} T. Ochiai, Euler system for Galois deformations, {\it Ann Inst. Fourier  Grenoble} {\bf 55}, 1  (2005) 113-146.
\bibitem[Oc4]{o4} T. Ochiai, Several Variables $p$-adic $L$-Functions for Hida Families of Hilbert Modular Forms, {\it Documenta Mathematica} {\bf 17}  (2012) 783-825.

\bibitem[M-T-T]{mtt} B. Mazur, J. Tate and J. Teitelbaum, On the $p$-adic analogues of the conjectures of Birch and Swinnerton-Dyer, {\emph Inventiones  Math.} {\bf 81} (1986) 1-48.
\bibitem[Pe]{pr} B. Perrin-Riou, Groupes de Selmer et Accouplements; Cas Particulier des Courbes Elliptiques, \emph{Documenta Mathematica}, Extra Volume Kato (2003) 725-760.

\bibitem[Sa]{ba} D. B. Salazar, Cohomologie sur convergente des variŽtiŽs modulaires de Hilbert et     
fonctions {$L$} {$p$}-adiques, PhD Thesis, UniversitŽ Lille1 (2013).

\bibitem[Su]{su} R. Sujatha, Iwasawa theory and modular forms, \emph{Pure Appl. Math. Q.} {\bf 2} (2006), no. 2, part 2, 519-538.

\bibitem[S-U]{es} C. Skinner and E. Urban,  The Iwasawa main conjectures  for $GL_2$, {\it Inventiones Math.} {\bf 195}(1) (2014) 1-277.
\bibitem[Ve]{ve}  O. Venjakob, On the structure theory of the Iwasawa algebra of a $p$-adic Lie group, \emph{J. Eur. Math. Soc.}  {\bf 4} (2002), no. 3, 271-311. 

\bibitem[Wa]{xw} X. Wan, Iwasawa Main Conjecture for Hilbert Modular Forms, Preprint, (2013).

\bibitem[Wi1]{wi} A. Wiles, On ordinary $\lambda$-adic representations associated to modular forms,  \emph{Invent. Math.}  {\bf 94}, no. 3,
(1988)   529-573.

\bibitem[Wi2]{wi2} A. Wiles, The Iwasawa conjecture for totally real fields.  \emph{Ann. of Math.} (2) {\bf 131}, no. 3, (1990) 493-540. 



               












\end{thebibliography}
\end{document}